%% file: main.tex
\documentclass[11pt]{amsart}
\usepackage{tikz,amsmath,amssymb,amsthm,amsfonts,genyoungtabtikz,todonotes,mathtools,float,mathrsfs}
\usetikzlibrary{knots}
\usetikzlibrary{cd}
\usetikzlibrary{patterns,decorations.pathreplacing}
\usepackage{nomencl}
\makenomenclature
\YFrench

\usepackage{hyperref}

\hypersetup{pdftitle={Super Multiset RSK and a Mixed Multiset Partition Algebra},pdfauthor={Alexander Wilson}} 
\hypersetup{colorlinks=true,linkcolor=blue,anchorcolor=blue,citecolor=blue}

\input{macros.tex}

\usepackage[nameinlink]{cleveref}

\newtheorem{prop}{Proposition}[section]
\newtheorem{theorem}[prop]{Theorem}
\newtheorem{lemma}[prop]{Lemma}
\newtheorem{corollary}[prop]{Corollary}
\theoremstyle{definition}
\newtheorem{definition}[prop]{Definition}
\theoremstyle{remark}
\newtheorem{remark}[prop]{Remark}
\newtheorem{example}[prop]{Example}

\title{Super Multiset RSK and a Mixed Multiset Partition Algebra}
\author{Alexander Wilson}

\begin{document}

\maketitle

\textbf{Abstract} Through dualities on representations on tensor powers and symmetric powers respectively, the partition algebra and multiset partition algebra have been used to study long-standing questions in the representation theory of the symmetric group. In this paper we extend this story to exterior powers, introducing the mixed multiset partition algebra as well as a generalization of the Robinson-Schensted-Knuth algorithm to two-row arrays of multisets with elements from two alphabets. From this algorithm, we obtain enumerative results which reflect representation-theoretic decompositions of this algebra. Furthermore, we use the generalized RSK algorithm to describe the decomposition of a polynomial ring in sets of commuting and anti-commuting variables as a module over both the general linear group and the symmetric group.

\section{Introduction}

Suppose an $A\times B$-module $V$ has a decomposition of the form \begin{align}
    V\cong\bigoplus_{\lambda\in\Lambda} A^\lambda\otimes B^\lambda.\label{eq:decomp}
\end{align} Such decompositions appear naturally in contexts like mutually centralizing actions or decomposing an algebra $A$ as an $A\times A$-module over itself. Comparing dimensions, this decomposition manifests enumeratively as
\begin{align*}
    \dim(V)&=\sum_{\lambda\in\Lambda} \dim(A^\lambda)\cdot\dim(B^\lambda).
\end{align*}

This equality implies that there exists a bijection \begin{align}
    \mathcal{V}\stackrel{\sim}{\longleftrightarrow}\biguplus_{\lambda\in\Lambda} \mathcal{A}_\lambda\times \mathcal{B}_\lambda. \label{eq:general_bijection}
\end{align} where $\mathcal{V}$, $\mathcal{A}_\lambda$, and $\mathcal{B}_\lambda$ are indexing sets for bases of $V$,  $A^\lambda$, and $B^\lambda$ respectively.

In \cite{robinson1938representations} Robinson (and later Schensted in \cite{schensted1961longest}) described a bijection between permutations in $S_n$ and pairs of standard Young tableaux, which explicitly provides the above bijection when $V=\C S_n$ is decomposed as an $S_n\times S_n$ module. In \cite{knuth1970permutations} Knuth generalized this process to one which takes in a biword and returns a pair of semistandard Young tableaux of the same shape, and this generalization is usually called the RSK algorithm. The RSK algorithm and its further generalizations provide bijections like \Cref{eq:general_bijection} explicitly in a number of contexts.

Generalizations of RSK also appear as bijections between pairs of bases of some associative algebras and Lie algebras. Super RSK describes a correspondence between objects called restricted superbiwords and semistandard supertableaux (see \cite{clark2015super,du2011quantum,grosshans1987invariant,scala2006super,marko2011bideterminants} for some examples of these combinatorial objects arising in the study of superalgebras and their representations). Although there are multiple bijections between these objects (see \cite{bonetti1988rs,shimozono2001color} for example), the one best suited to our purposes turns out to be that introduced in \cite{Muth2019Super}.

A variant of RSK which inserts arrays of multisets was introduced in \cite{Colmenarejo2020Insertion}, in which the authors used this algorithm to describe bijections like \Cref{eq:general_bijection} for decompositions of diagram algebras such as the partition algebra (introduced in \cite{jones1993potts,martin1994temperley}) and the multiset partition algebra (introduced in \cite{narayanan2019multiset,orellana2020howe}).

Diagram algebras have been used to address long-standing questions about the representation theory of the symmetric group. For example: 

\vspace{-.05in}

\begin{itemize}
    \item In \cite{bowman2015partition}, the authors use the partition algebra to compute a positive formula for some Kronecker coefficients.
    \item In \cite{orellana2021hopf,Orellana2015SymmetricGC} the authors use the partition algebra to introduce a basis of symmetric functions in order to study reduced Kronecker coefficients and restriction coefficients.
    \item In \cite{orellana2021uniform}, the authors study the uniform block permutation algebra and the connections between its characters and the plethysm operation on symmetric functions.
\end{itemize}

\vspace{-.05in}

This paper is organized as follows:

In \Cref{sec:emp_insertion}, we adapt the insertion algorithm for multiset partitions introduced in \cite{Colmenarejo2020Insertion} by replacing the underlying RSK algorithm with the super RSK algorithm of \cite{Muth2019Super}.

In \Cref{sec:epa}, we introduce the mixed multiset partition algebra $\MP_{\bfa,\bfb}(x)$, which generalizes the multiset partition algebras of \cite{narayanan2019multiset} and \cite{orellana2020howe}. We introduce a basis for the algebra and describe a diagrammatic product on this basis. When $x$ is specialized to an integer $n\geq2\abs{\bfa}+2\abs\bfb$, the algebra $\MP_{\bfa,\bfb}(n)$ is isomorphic to the centralizer $\End_{S_n}(\Sym^\bfa(V_n)\otimes\bigwedge^\bfb V_n)$ where $V_n$ is an $n$-dimension complex vector space.

In \Cref{sec:epa_reps} we construct the simple modules $\MP_{\bfa,\bfb}^\lambda$ of $\MP_{\bfa,\bfb}(n)$ for $n\geq2\abs{\bfa}+2\abs\bfb$. The super multiset partition insertion algorithm of \Cref{sec:emp_insertion} is an enumerative manifestation of the decomposition \begin{align*}
    \MP_{\bfa,\bfb}(n)\cong\bigoplus_{\lambda\in\Lambda^{\MP_{\bfa,\bfb}(n)}}(\MP_{\bfa,\bfb}^\lambda)^*\otimes\MP_{\bfa,\bfb}^\lambda
\end{align*} and as such is employed in this section to prove that a spanning set for $\MP_{\bfa,\bfb}^\lambda$ is in fact a basis.

Finally, in \Cref{sec:conclusions} we use these results to recover a combinatorial interpretation for the decomposition of a multivariate polynomial ring as an $S_n$-module given in \cite{orellana2020combinatorial} as well as provide an interpretation for its decomposition as a $GL_n$-module using analogous tableaux.

\section{Preliminaries and Definitions}
\subsection{Set Partitions and Multiset Partitions}

A \defn{set partition} $\rho$ of a set $S$ is a set of nonempty subsets of $S$ called \defn{blocks} whose disjoint union is $S$. We write $\ell(\rho)$ for the number of blocks of $\rho$. A \defn{multiset} of size $r$ from a set $S$ is a collection of $r$ unordered elements of $S$ which can be repeated. We will write multisets in $\multi{,}$ to differentiate them from sets and we will usually denote them by a capital letter with a tilde. We may write multisets using exponential notation $\tS=\multi{{s_1}^{m_1},\dots,{s_k}^{m_k}}$ where the multiplicity of the element $s_i$ is given by the exponent $m_i$. We write $m_{s_i}(\tS)=m_i$ for this multiplicity. Given multisets $\tS=\multi{{s_1}^{m_1},\dots,{s_k}^{m_k}}$ and $\tR=\multi{{s_1}^{n_1},\dots,{s_k}^{n_k}}$, write $\tS\uplus\tR=\multi{{s_1}^{m_1+n_1},\dots,{s_k}^{m_k+n_k}}$ for the union of the two multisets. A \defn{multiset partition} $\tilde\rho$ of a multiset $\tS$ is a multiset of multisets called blocks whose union is $\tS$. We write $\ell(\tilde\rho)$ for the number of blocks of $\trho$.

We will be interested in set and multiset partitions with elements from the following four alphabets of positive integers: \begin{align*}
    [r]=\{1,\dots,r\} && [\ov r]=\{\ov1,\dots,\ov r\}\\
    [\un r]=\{\un1,\dots,\un r\} && [\ovun r]=\{\ovun1,\dots,\ovun r\}
\end{align*}

We call the numbers $[\ov r]$ and $[\ovun r]$ \defn{barred}, and we call the numbers $[\un r]$ and $[\ovun r]$ \defn{underlined}. Within each alphabet we order the numbers as usual, and between alphabets we adopt the convention that  \[i<\ov j < \un k < \ovun m\] for all $i,j,k,$ and $m$.

A \defn{weak composition} of an integer $r$ of length $k$ is a sequence of $k$ non-negative integers which sum to $r$. Write $W_{r,k}$ for the set of weak compositions of $r$ of length $k$. For $\bfa\in W_{r,k}$, write $\bfa_i$ for the $i$th number in the sequence. We say that $\bfa$ is less than or equal to $\bfb$ in the \defn{dominance order}, written $\bfa\trianglelefteq\bfb$, if \[\bfa_1+\dots+\bfa_i\leq \bfb_1+\dots+\bfb_i\] for all $1\leq i\leq k$. For $\bfa\in W_{r,k}$, write $[k]^{\bfa}=\multi{1^{\bfa_1},\dots,k^{\bfa_k}}$.

For $\bfa\in W_{r,k}$ and $\bfb\in W_{s,m}$, we define the following sets of set and multiset partitions where in each set we assume that no barred entry is ever repeated in a block of a multiset partition. Let 
\begin{itemize}
    \item $\Pi_{2r}$ be the set of set partitions of $[r]\cup[\un r]$.
    \item $\Pi_{2(\bfa,\bfb)}$ be the set of multiset partitions of $[k]^\bfa\uplus[\ov m]^\bfb\uplus[\un k]^\bfa\uplus [\ovun m]^\bfb$.
    \item $\hat\Pi_{2(\bfa,\bfb)}$ be the set of multiset partitions in $\Pi_{2(\bfa,\bfb)}$ where repeated blocks have an even number of barred entries. We call these \defn{restricted multiset partitions}.
\end{itemize}

\begin{example}
    \begin{align*}
        \{\{1,\un1\},\{2,3,\un2,\un3\}\}&\in\Pi_{2(3)}\\
        \multi{\multi{1,\ov1,\un1,\un2},\multi{2,\ov1},\multi{\ovun1},\multi{\ovun1}}&\in\Pi_{2((1,1),(2,0))}\\
        \multi{\multi{1,\un1,\un2},\multi{2},\multi{\ov1,\ovun1},\multi{\ov1,\ovun1}}&\in\hat\Pi_{2((1,1),(2,0))}
    \end{align*}
\end{example}

A block of a set or multiset partition is called \defn{propagating} if it contains both an underlined element and a non-underlined element.

Let $\tS$ and $\tR$ be multisets from $[k]$. We say that $\tS$ is less than $\tR$ in the \defn{last-letter order}, written $\tS<\tR$, if one of the following conditions hold: (i) $\tS$ is empty and $\tR$ is not, (ii) $\max(\tS)<\max(\tR)$, or (iii) $\max(\tS)=\max(\tR)=u$ and $\tS\smallsetminus\{u\}<\tR\smallsetminus\{u\}$.

For $\bfa\in W_{r,k}$ and $\bfb\in W_{s,m}$, define the \defn{coloring map} $\kappa_{\bfa,\bfb}:[r+s]\to [k]\cup[\ov m]$ by \[\kappa_{\bfa,\bfb}(i)=\begin{cases}
    1 & 1\leq i \leq \bfa_1\\
    2 & \bfa_1<i\leq \bfa_1+\bfa_2\\
    \vdots & \\
    k & \bfa_1+\dots+\bfa_{k-1}<i\leq r\\
    \ov1 & r < i \leq r + \bfb_1\\
    \vdots & \\
    \ov m & r+\bfb_1+\dots+\bfb_{m-1} < i \leq r+s
\end{cases}.\] We can extend this map to a map from subsets of $[r+s]$ to multisets with elements in $[k]\cup[\ov m]$ where $\kappa_{\bfa,\bfb}(\{i_1,\dots,i_\ell\})=\multi{\kappa_{\bfa,\bfb}(i_1),\dots,\kappa_{\bfa,\bfb}(i_\ell)}$. Finally, we can extend $\kappa_{\bfa,\bfb}$ to $\Pi_{2(r+s)}$ by applying $\kappa_{\bfa,\bfb}$ to each block where we set $\kappa_{\bfa,\bfb}(\un i)=\un{\kappa_{\bfa,\bfb}(i)}$. Note that each $\tpi\in\Pi_{2(\bfa,\bfb)}$ is in the image $\kappa_{\bfa,\bfb}(\Pi_{2r})$, but this image also contains multiset partitions with blocks having repeated barred elements.

\begin{example}\label{ex:coloring_sp} Here we give an example of a set partition which does map to an element of $\Pi_{2(\bfa,\bfb)}$ under the coloring map.
    \begin{align*}
        \kappa_{(2,1),(0,2)}&\left(\{\{1,2,5\},\{3,4,\un3\},\{\un2,\un4\},\{\un1,\un5\}\}\right)\\
        &=\multi{\multi{1,1,\ov2},\multi{2,\ov2,\un2},\multi{\un1,\ovun2},\multi{\un1,\ovun2}}
    \end{align*}
\end{example}

Let $\tpi\in\Pi_{2(\bfa,\bfb)}$ and let $\pi\in\kappa_{\bfa,\bfb}^{-1}(\tpi)$ be such that if $\kappa_{\bfa,\bfb}(i)=\kappa_{\bfa,\bfb}(j)$ and $i<j$, then the block containing $i$ in $\pi$ is weakly before the block containing $j$ in the last-letter order. The unique such element is called the \defn{standardization} of $\tpi$.

\begin{example} The standardization of a multiset partition can be computed by first putting its blocks in last-letter order, then for each value, replacing it with its pre-image under $\kappa_{\bfa,\bfb}$ increasing left-to-right.
    \input{figures.ex_standardized_sp}
\end{example}

\begin{remark}\label{rmk:standard_content} If $\pi=\{A_1,\dots,A_\ell\}$ is the standardization of $\kappa_{\bfa,\bfb}(\pi)=\tpi$, then its blocks satisfy \[A_i<A_j\implies\kappa_{\bfa,\bfb}(A_i)\leq\kappa_{\bfa,\bfb}(A_j).\]
\end{remark}

\subsection{Diagrams}

For a set partition $\pi\in\Pi_{2r}$, there is a classical graph-theoretic representation of $\pi$ on two rows of vertices with the top row being labeled $1$ through $r$ and the bottom being labeled $\un1$ through $\un{r}$. Two vertices of this graph are connected if and only if their labels are in the same block of $\pi$. 

\begin{example} The set partition $\pi=\{\{1,\un1,\un2,\un3\},\{2,3\},\{\un4\},\{4,5,\un5\}\}$ could be represented by either of the following two graphs.
    \begin{align*}
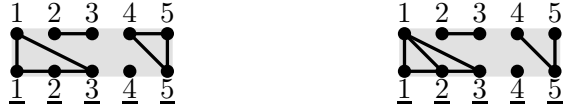
        \begin{array}{c}
            \begin{tikzpicture}[xscale=.5,yscale=.5,line width=1.25pt] 
                \foreach \i in {1,2,3,4,5}  { \path (\i,1.25) coordinate (T\i); \path(\i,1.8) node {$\i$}; \path (\i,.25) coordinate (B\i); \path(\i,-.3) node {$\un{\i}$}; } 
                \filldraw[fill= black!12,draw=black!12,line width=4pt]  (T1) -- (T5) -- (B5) -- (B1) -- (T1);
                \draw[black] (T1)--(B3)--(B2)--(B1)--(T1);
                \draw[black] (T2)--(T3);
                \draw[black] (T4)--(T5)--(B5)--(T4);
                \colortop{1,2,3,4,5}{black}
                \colorbot{1,2,3,4,5}{black}
            \end{tikzpicture}
        \end{array} && \begin{array}{c}
            \begin{tikzpicture}[xscale=.5,yscale=.5,line width=1.25pt] 
                \foreach \i in {1,2,3,4,5}  { \path (\i,1.25) coordinate (T\i); \path(\i,1.8) node {$\i$}; \path (\i,.25) coordinate (B\i); \path(\i,-.3) node {$\un{\i}$}; } 
                \filldraw[fill= black!12,draw=black!12,line width=4pt]  (T1) -- (T5) -- (B5) -- (B1) -- (T1);
                \draw[black] (T1)--(B3)--(B2)--(B1)--(T1);
                \draw[black] (T1)--(B2);
                \draw[black] (T2)--(T3);
                \draw[black] (T5)--(B5)--(T4);
                \colortop{1,2,3,4,5}{black}
                \colorbot{1,2,3,4,5}{black}
            \end{tikzpicture}
        \end{array}
    \end{align*}
\end{example}

Because there are many graphs that represent the set partition $\pi$, we consider two graphs equivalent if their connected components give rise to the same set partition. The \defn{diagram} of $\pi$ is the equivalence class of graphs with the same connected components.

We can similarly consider a graph-theoretic representation of any multiset partition $\tpi\in\Pi_{2(\bfa,\bfb)}$ with $\bfa\in W_{r,k}$ and $\bfb\in W_{s,m}$. This time we place $r+s$ vertices on the top labeled by $[k]^\bfa\uplus[\ov m]^\bfb$ in weakly increasing order and place $r+s$ vertices on the bottom labeled by $[\un k]^\bfa\uplus[\ovun m]^\bfb$ in weakly increasing order. We then connect the vertices in any way so that the labeled connected components taken together are $\tpi$.

\begin{example}\label{MSPDiagramExample} The multiset partition \[\tpi=\multi{\multi{1,1},\multi{1,\un1,\un1,\un1},\multi{\ov1},\multi{2,\ov1,\ovun1},\multi{\un2,\ovun1}}\] could be represented by either of the following graphs:
    \begin{align*}
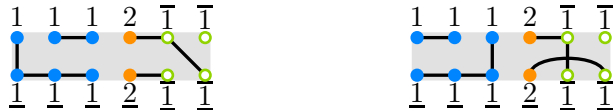
    \begin{array}{c}
            \begin{tikzpicture}[xscale=.5,yscale=.5,line width=1.25pt] 
                \foreach \i in {1,2,3,4,5,6}  { \path (\i,1.25) coordinate (T\i); \path (\i,.25) coordinate (B\i); } 
                \path(1,1.8) node {$1$};
                \path(2,1.8) node {$1$};
                \path(3,1.8) node {$1$};
                \path(4,1.8) node {$2$};
                \path(5,1.8) node {$\ov1$};
                \path(6,1.8) node {$\ov1$};
                \path(1,-.3) node {$\un1$};
                \path(2,-.3) node {$\un1$};
                \path(3,-.3) node {$\un1$};
                \path(4,-.3) node {$\un2$};
                \path(5,-.3) node {$\ovun1$};
                \path(6,-.3) node {$\ovun1$};
                \filldraw[fill= black!12,draw=black!12,line width=4pt]  (T1) -- (T6) -- (B6) -- (B1) -- (T1);
                \draw[black] (B3)--(B2)--(B1)--(T1);
                \draw[black] (T2)--(T3);
                \draw[black] (T4)--(T5)--(B6);
                \draw[black] (B4)--(B5);
                \colortop{1,2,3}{c1}
                \colortop{4}{c2}
                \colortopalt{5,6}{c3}
                \colorbot{1,2,3}{c1}
                \colorbot{4}{c2}
                \colorbotalt{5,6}{c3}
            \end{tikzpicture}
        \end{array} && \begin{array}{c}
            \begin{tikzpicture}[xscale=.5,yscale=.5,line width=1.25pt] 
                \foreach \i in {1,2,3,4,5,6}  { \path (\i,1.25) coordinate (T\i); \path (\i,.25) coordinate (B\i); } 
                \path(1,1.8) node {$1$};
                \path(2,1.8) node {$1$};
                \path(3,1.8) node {$1$};
                \path(4,1.8) node {$2$};
                \path(5,1.8) node {$\ov1$};
                \path(6,1.8) node {$\ov1$};
                \path(1,-.3) node {$\un1$};
                \path(2,-.3) node {$\un1$};
                \path(3,-.3) node {$\un1$};
                \path(4,-.3) node {$\un2$};
                \path(5,-.3) node {$\ovun1$};
                \path(6,-.3) node {$\ovun1$};
                \filldraw[fill= black!12,draw=black!12,line width=4pt]  (T1) -- (T6) -- (B6) -- (B1) -- (T1);
                \draw[black] (B1)--(B2)--(B3)--(T3);
                \draw[black] (T1)--(T2);
                \draw[black] (T4)--(T5)--(B5);
                \draw[black] (B4)  .. controls +(0,+.6) and +(0,+.6) .. (B6);
                \colortop{1,2,3}{c1}
                \colortop{4}{c2}
                \colortopalt{5,6}{c3}
                \colorbot{1,2,3}{c1}
                \colorbot{4}{c2}
                \colorbotalt{5,6}{c3}
            \end{tikzpicture}
        \end{array}
    \end{align*}
\end{example}

Again the \defn{diagram} of $\tpi$ is the equivalence class of graphs whose labeled connected components give $\tpi$.

We will often drop the labels on these graphs. A set partition diagram will be distinguished by the black color of its vertices, and a multiset partition diagram will be distinguished by its colored vertices. Its vertices are understood to be labeled with \textcolor{c1}{\textbf{blue}}, \textcolor{c2}{\textbf{orange}}, \textcolor{c3}{\textbf{green}}, and \textcolor{c4}{\textbf{purple}} representing 1, 2, $\ov1$, and $\ov2$ respectively, where the vertices labeled with barred elements are further distinguished by being drawn as open circles.

The multiset partition diagram for $\tpi\in\Pi_{2(\bfa,\bfb)}$ drawn in \defn{standard form} is the diagram for the standardization $\pi$ of $\tpi$ with vertices labeled $i$ colored with the color $\kappa_{\bfa,\bfb}(i)$. The right-most diagram in \Cref{MSPDiagramExample} is drawn in standard form.

\subsection{Tableaux}

A \defn{partition} of $n$ is a weakly-decreasing sequence $\lambda$ of positive integers called parts which sum to $n$. We write $\lambda\vdash n$ to mean that $\lambda$ is a partition of $n$. We write $\lambda_i$ for the $i$th element of the sequence $\lambda$, called the $i$th part of $\lambda$. Given a partition $\lambda$, its \defn{Young diagram} is an array of left-justified boxes where the $i$th row from the bottom has $\lambda_i$ boxes. For example, the Young diagram of $(5,3,3,1)\vdash 12$ is \begin{align*}
    \yng(5,3,3,1).
\end{align*}

When we refer to the $i$th row of a Young diagram, we mean the $i$th row from the bottom, which corresponds to the $i$th part of $\lambda$.

A \defn{standard Young tableau} $t$ of shape $\lambda\vdash n$ is a filling of these boxes with the numbers $[n]$ whose rows increase left-to-right and columns increase bottom-to-top. Write $\SYT(\lambda)$ for the set of standard Young tableaux of shape $\lambda$.

A \defn{semistandard Young tableau} is a filling of the Young diagram of $\lambda$ with positive integers with rows increasing left-to-right and columns increasing \textit{weakly} bottom-to-top. Write $\SSYT(\lambda,r)$ for the set of semistandard Young tableaux of shape $\lambda$ and maximum entry $r$.

A \defn{multiset partition tableau} $\tT$ of shape $\lambda$ is a filling of the Young diagram of $\lambda$ with multisets along with at least $\lambda_2$ empty boxes in the first row (so that no nonempty box in the first row is adjacent to a box in the second row). The \defn{content} of a tableau $\tT$ is the multiset partition consisting of the contents of non-empty boxes of $\tT$. For $\bfa\in W_{r,k}$ and $\bfb\in W_{s,m}$, write $\MT(\lambda,\bfa,\bfb)$ for the set of multiset partition tableaux of shape $\lambda$ with content a multiset partition of $[k]^\bfa\uplus[\ov m]^\bfb$.

A \defn{standard multiset partition tableau} is a multiset partition tableau $T$ whose content is a set partition and whose rows increase left-to-right and columns increase bottom-to-top with respect to the last-letter order. Write $\SMT(\lambda,r)$ for the set of standard multiset partition tableaux of shape $\lambda$ and content a set partition of $[r]$.

Finally, a \defn{semistandard multiset partition tableau} $\hat T$ is a multiset partition tableau such that:
\begin{enumerate}
    \item The rows and columns of $\hat T$ weakly increase in the last-letter order
    \item For two identical blocks $B$ and $C$ of the content of $\hat T$,
    \begin{enumerate}
        \item If the blocks contain an even number of barred entries, then $B$ and $C$ cannot be placed in the same column.
        \item If the blocks contain an odd number of barred entries, then $B$ and $C$ cannot be placed in the same row.
    \end{enumerate}
\end{enumerate}

Write $\SSMT(\lambda,\bfa,\bfb)$ for the set of semistandard multiset partition tableaux of shape $\lambda$ with content a multiset partition of $[k]^\bfa\uplus[\ov m]^\bfb$. 

\begin{example} Examples of these classes of tableaux of the same shape $\lambda=(5,2,1)$:
    \begin{align*}
        \begin{array}{ll}
            \inline{\young(12478,35,6)}\!\!\!\in\SYT(\lambda) & \inline{\young(11233,24,3)}\!\!\!\in\SSYT(\lambda, 4) \\ \\ \inline{\young(\;\;\;<24><9>,<35><68>,<17>)}\!\!\!\in\SMT(\lambda,9) &
            \inline{\young(\;\;\;<\ov1\ov2><\ov1\ov2>,<1\ov2>2,<1\ov2>)}\!\!\!\in\SSMT(\lambda,(2,1),(2,4))
        \end{array}
    \end{align*}
\end{example}

Given $T\in\SMT(\lambda,r+s)$ and $\bfa\in W_{r,k},\bfb\in W_{s,m}$, write $\kappa_{\bfa,\bfb}(T)$ for the result of applying $\kappa_{\bfa,\bfb}$ to the content of each box of $T$. For a multiset partition tableau $\tT\in\MT(\lambda,\bfa,\bfb)$, define its \defn{standardization} to be the unique $T\in\kappa_{\bfa,\bfb}^{-1}(\tT)$ whose content is the standardization of the content of $T$ such that for any two boxes with contents $A<B$ with $\kappa_{\bfa,\bfb}(A)=\kappa_{\bfa,\bfb}(B)$, the box containing $A$ is weakly to the right and weakly below the box containing $B$.

\begin{example} Here we compute the standardization $T$ of a semistandard multiset partition tableau $\tT$. At each step, we simultaneously replace all boxes with a particular content, proceeding in the last-letter order.
    \input{figures.ex_standardization.tex}
\end{example}

\begin{definition}
    For multiset partition tableaux $\tT$ and $\tS$, let $\{\tB_1,\dots,\tB_\ell\}$ be the set of distinct contents of the boxes of $\tT$ and $\tS$ in last-letter order. For $1\leq i\leq \ell$, define the weak composition $\alpha^{\tB_i}(\tT)$ by \begin{align*}
        \alpha^{\tB_i}(\tT)_j=\#\text{boxes with content $\leq \tB_i$ in column $j$}.
    \end{align*}

    We say that $\tT$ is larger than $\tS$ in the \defn{column dominance order}, written $\tS\triangleleft \tT$, if $\alpha^{\tB_i}(\tS)\triangleleft\alpha^{\tB_i}(\tT)$ for all $1\leq i\leq\ell$.
\end{definition}

\begin{example} Two semistandard multiset partition tableaux compared in the column-dominance order:
    \input{figures.ex_column_dominance}
\end{example}

\begin{remark}\label{rmk:col_dominance}
    If the content of $T$ is the standardization of a multiset partition as in \Cref{rmk:standard_content}, then $\alpha^{\tilde B}(\kappa_{\bfa,\bfb}(T))=\alpha^B(T)$ where $B$ is the largest block of the content of $T$ such that $\kappa_{\bfa,\bfb}(B)=\tilde{B}$. It's then clear to see that if $T$ and $S$ have this property and $T\lhd S$, then $\kappa_{\bfa,\bfb}(T)\lhd\kappa_{\bfa,\bfb}(S)$.
\end{remark}

\begin{example} Here we show the composition sequence for a semistandard multiset partition tableau along with that of its standardization:
    \input{figures.ex_compositions_and_standardization}
\end{example}

\subsection{Symmetric Functions}

Symmetric functions are a powerful tool for studying the representation theory of $S_n$ and $GL_n$. The ring of symmetric functions consists of formal power series in an infinite number of variables $X=\{x_1,x_2,\dots\}$ which are fixed under any permutation of a finite number of the variables. A special class of symmetric functions are the \defn{Schur functions} $s_\lambda$ for $\lambda$ a partition. The Schur functions are defined by \begin{align*}
    s_\lambda=\sum_{T} x^T
\end{align*} where the sum is taken over all semistandard Young tableaux of shape $\lambda$ and $x^T$ is the monomial where the exponent of $x_i$ is the number of times $i$ occurs in the tableau $T$.

When all but the first $n$ variables in the Schur function $s_\lambda$ are set to zero, we obtain a polynomial $s_\lambda(X_n)$ in the variables $X_n=\{x_1,\dots,x_n\}$. Given a matrix $A\in GL_n$, evaluating $s_\lambda(X_n)$ at the eigenvalues of $A$ gives the character value for the simple polynomial representation $GL_n^\lambda$. Using this fact, we will use symmetric function identities in \Cref{sec:GLn_module} to describe the decomposition of a polynomial $GL_n$-module into simple modules. To that end, we introduce two special cases of Schur functions and the simple modules they compute the characters of: The \defn{elementary symmetric function} \begin{align*}
    e_k=s_{(1^k)}=\sum_{i_1<i_2<\dots<i_k} x_{i_1}x_{i_2}\cdots x_{i_k}
\end{align*} corresponds to the exterior power $\bigwedge\nolimits^k V_n$ and the \defn{complete homogeneous symmetric function} \begin{align*}
    h_k=s_{(k)}=\sum_{i_1\leq i_2\leq\dots\leq i_k} x_{i_1}x_{i_2}\cdots x_{i_k}
\end{align*} corresponds to the symmetric power $\Sym^k(V_n)$.

\subsection{Mutual Centralizers and the Partition Algebra}

One natural place that decompositions like (\ref{eq:decomp}) appear is the setting of mutually centralizing actions. Let $V_n=\C^n$ be an $n$-dimensional vector space over $\C$. A prototypical example of mutually centralizing actions is classical Schur--Weyl duality of the general linear group $GL_n$ and the symmetric group $S_r$ acting on the tensor power ${V_n}^{\otimes r}$. The general situation is given in the following theorem called the double centralizer theorem. For a semisimple algebra $A$, write \defn{$\Lambda^A$} for an indexing set of simple $A$-modules.

\begin{theorem} \cite[Section 6.2.5]{procesi}\cite[Section 4.2.1]{gw}\label{thm:dct}Let $A$ be a semisimple algebra acting faithfully on a module $V$ and set $B=\End_A(V)$. Then $B$ is semisimple and $\End_B(V)\cong A$. Furthermore, there is a set $P$ (a subset of the indexing set of the simple $A$-modules) such that for each $x\in P$, $A^x$ is a simple $A$-module occurring in the decomposition of $V$ as an $A$-module. If we set $B^x=\Hom(A^x, V)$, then $B^x$ is a simple $B$-module and the decomposition of $V$ as an $A\times B$-module is \[V\cong\bigoplus_{x\in P}A^x\otimes B^x.\] Moreover, the dimension of $A^x$ is equal to the multiplicity of $B^x$ in $V$ as a $B$-module and the dimension of $B^x$ is equal to the multiplicity of $A^x$ in $V$ as an $A$-module.
\end{theorem}

We are interested in applying the double centralizer theorem to the following situation. Let \[e\in\End_{GL_n}({V_n}^{\otimes r})\cong\C S_r\] be an idempotent. Then $e{V_n}^{\otimes r}$ is a $GL_n$-module, and hence by restriction to the $n\times n$ permutation matrices an $S_n$-module. Then \[B=\End_{S_n}(e{V_n}^{\otimes r})\cong e\End_{S_n}({V_n}^{\otimes r})e\] is a semisimple algebra. The following is a well-known fact (see e.g. \cite[Theorem 1.10.14]{linckelmann_2018}), which will help us study the simple $B$-modules.

\begin{prop} \label{prop:projected_irreps} Let $A$ be an algebra and $e\in A$ an idempotent. Then $eAe$ is an algebra with identity $e$, and the following can be said about its simple modules.
\begin{enumerate}
    \item If $S$ is a simple $A$-module, then $eS$ is either zero or a simple $eAe$-module.
    \item All simple $eAe$-modules arise in this way. That is, if $T$ is a simple $eAe$-module, then there is a simple $A$-module $S$ such that $eS\cong T$.
\end{enumerate}
\end{prop}

When $n\geq 2r$, the centralizer algebra $\End_{S_n}({V_n}^{\otimes r})$ is isomorphic to the Partition algebra $P_r(n)$ introduced by Jones \cite{jones1993potts} and Martin \cite{martin1994temperley} as a generalization of the Temperley-Lieb algebra and the Potts model in statistical mechanics, so the algebra $B$ is given by $B\cong e P_r(n) e$. We devote the remainder of this section to describing the structure of $P_r(n)$ and its simple modules for $n\geq 2r$.

\subsubsection{Diagram Basis of $P_r(x)$ and $S_r$ actions}
For $r$ a positive integer and an indeterminate $x$, the partition algebra $P_r(x)$ is an associative algebra over $\C(x)$. When $x$ is specialized to an integer $n\geq 2r$, the algebra $P_r(n)$ is isomorphic to the algebra of endomorphisms $\End_{S_n}({V_n}^{\otimes r})$ of ${V_n}^{\otimes r}$ as an $S_n$-module. The partition algebra has a basis called the \defn{diagram basis} $\{\LL_\pi:\pi\in\Pi_{2r}\}$ whose product has a combinatorial interpretation in terms of partition diagrams. To compute the product of $\LL_\pi$ and $\LL_\nu$, place a graph representing $\pi$ above one representing $\nu$ and identify the vertices on the bottom of $\pi$ with the corresponding vertices of $\nu$ to create a three-tiered diagram. Let $\pi\circ\nu$ be the restriction of this diagram to the very top and very bottom, preserving which vertices are connected in the larger diagram and let $c(\pi,\nu)$ be the number of components entirely in the middle of the three-tiered diagram. Then $\LL_\pi\LL_\nu=x^{c(\pi,\nu)}\LL_{\pi\circ\nu}$.

\begin{example} Here we show the product of two diagram basis elements. Notice that two components are entirely in the middle, giving a coefficient of $x^2$.
    \input{figures.example_diagram_basis.tex}
\end{example}

The symmetric group $S_r$ sits naturally inside $P_r(x)$ where the permutation $\sigma\in S_r$ corresponds to the set partition $\{\{\sigma(1),\un1\},\dots,\{\sigma(r),\un{r}\}\}$. Writing $\LL_\sigma$ for the corresponding diagram-basis element, the product $\LL_{\sigma_1}\LL_\pi\LL_{\sigma_2}$ leads to an action of $S_r\times S_r$ on $\Pi_{2r}$ where we write $\sigma_1.\pi.\sigma_2$ for the resulting set partition so that \[\LL_{\sigma_1}\LL_\pi\LL_{\sigma_2}=\LL_{\sigma_1.\pi.\sigma_2}.\] The set partition $\sigma_1.\pi.\sigma_2$ is obtained from $\pi$ by replacing each non-underlined element $i$ with $\sigma_1(i)$ and each underlined element $\un{i}$ with $\un{{\sigma_2}^{-1}(i)}$.

The following straightforward fact will be useful later.

\begin{lemma}\label{lem:fixing_partition} The subgroup $(S_r\times S_r)^\pi$ which fixes a set partition $\pi\in\Pi_{2r}$ decomposes as a semidirect product \begin{align*}
    (S_r\times S_r)^\pi&=\left(\prod_{B\in\pi} N_B\right)\rtimes H
\end{align*} where $N_B$ consists of permutations within the block $B$ and $H$ consists of permutations of whole blocks of $\pi$ with the same number of underlined and non-underlined elements.
\end{lemma}

\begin{example} For $\pi=\{\{1,2,\un1\},\{3,4,\un2\},\{5\},\{\un3,\un4,\un5\}\}$, the subgroup $(S_5\times S_5)^\pi$ of permutations that fix $\pi$ decomposes as above where \begin{align*}
    N_{\{1,2,\un1\}}&=S_{\{1,2\}}\times S_{\{\un1\}}\\
    N_{\{3,4,\un2\}}&=S_{\{3,4\}}\times S_{\{\un2\}}\\
    N_{\{5\}}&=S_{\{5\}}\\
    N_{\{\un3,\un4,\un5\}}&=S_{\{\un3,\un4,\un5\}}\\
    H&=\{\id, (1\,3)(2\,4)(\un1\,\un2)\}.
\end{align*}
\end{example}

\subsubsection{Representations of \texorpdfstring{$P_r(n)$}{Pr(n)}}

In \cite{Halverson2018SetpartitionTA}, the authors construct the simple modules $P_r^\lambda$ of $P_r(n)$ for $n\geq 2r$. For $\lambda\in\Lambda^{P_r(n)}$, the module $P_r^\lambda$ consists of formal linear combinations of standard multiset partition tableaux \[P_r^\lambda\defeq\C\{v_T:T\in\SMT(\lambda,r)\}\] with action given as follows. For a set partition $\pi\in\Pi_{2r}$ to act on a tableau $T$, first pull out the content of $T$, a set partition of $[r]$, into a single row. Then, put $\pi$ on top of this row and identify the corresponding vertices. Form $T'$ by replacing the content of each box in $T$ with the set of vertices at the top of $\pi$ that the box is connected to, and for each block entirely in the top of $\pi$, include it as the content of a box in the first row of $T'$. If two blocks above the first row are combined or the content of a box above the first row does not connect to the top of the partition diagram, the result is zero.

\begin{example} Here we show the action of two different diagrams on the same tableau.
    \input{figures.partition_action.tex}
\end{example}

The result $T$ of the above process may not have increasing rows and columns, so we need to make sense of what it means to write $v_T$ for $T$ nonstandard. The algorithm for writing $v_T$ as a linear combination of standard tableaux is called the \defn{straightening algorithm}. The straightening algorithm for $P_r^\lambda$ is the same as for the Specht modules of $S_n$ applied to the rows above the first row of $T$, a fact which is clear from the construction in \cite[Section 3.4]{Halverson2018SetpartitionTA}. A complete treatment of this algorithm for Specht modules can be found in \cite[Section 2.6]{sagan}, but we summarize the process here.

Let $\lambda\in\Lambda^{P_r(n)}$. Because $S_r$ sits inside $P_r(n)$, there is a natural $S_r$-action on $P_r^\lambda$ where $\sigma$ acts by replacing each value $i$ in $T\in\SMT(\lambda,r)$ with $\sigma(i)$. Setting $m=\lambda_2+\dots+\lambda_\ell$, there is also an action of $S_m$ by permuting the boxes above the first row of $T$.

The simplest step of the straightening algorithm involves permuting boxes within the same column. Let $T$ be a multiset partition tableau with content a set partition of $[n]$. If $T'$ is obtained from $T$ by swapping two boxes in the same column, then \[v_{T'}=-v_T.\] Applying this rule repeatedly allows us to assume that the tableau we want to straighten has increasing columns.

Now suppose $T$ is a multiset partition tableau of shape $\lambda$ with increasing columns. We say that $T$ has a \defn{decrease} at boxes $a$ and $b$ if the box $b$ is immediately to the right of $a$ in the same row and the content of $a$ is larger than the content of $b$. Suppose $T$ has a decrease at boxes $a$ and $b$ and let $A$ be the set of boxes including $a$ and each box above it and let $B$ be the set of boxes including $b$ and each box below it. The \defn{Garnir transversal} $g(A,B)$ consists of all permutations of the boxes in $A\cup B$ so that their contents still increase up columns of $T$.

The following theorem is an immediate consequence of applying \cite[Theorem 2.6.4]{sagan} to the construction of $P_r^\lambda$ in \cite{Halverson2018SetpartitionTA}.

\begin{theorem} \label{thm:straightening} Suppose $T$ is a multiset partition tableau with content a set partition of $[r]$ which has a decrease along a row. Let $g(A,B)$ be the Garnir transversal for the corresponding subsets $A$ and $B$ of boxes of $T$ as above. Then \begin{align*}
    \sum_{\sigma\in g(A,B)}\sgn(\sigma) v_{\sigma.T}=0
\end{align*} where $T\triangleleft \sigma.T$ for $\sigma\neq\id$.
\end{theorem}

Using \Cref{thm:straightening}, we can write any nonstandard $T$ as a linear combination of tableaux strictly larger in the column dominance order. Because the tableaux continue to increase in this order and there are only finitely many tableaux of a given content, this process eventually terminates, writing $v_T$ as a linear combination of $v_S$ for $S\in\SMT(\lambda,r)$.

\begin{example} We apply one iteration of the straightening algorithm to the following tableau:

    \input{figures.ex_straightening}

    Now we can solve for $v_T$ on the left as a linear combination of tableaux with the decrease eliminated.
    
\end{example}

\begin{remark}\label{rmk:straightening} Note that in the above example, we are considering an action of $S_5$ on the boxes above the first row rather than the usual action of $S_9$ on the content of the tableau. In  the case that the size of the contents of the boxes is preserved, we can embed this action of $S_5$ into $S_9$. For example, the permutation $(2\,3)(4\,5)\in S_5$ in the above example corresponds to $(4\,8)(3\,6)(5\,9)\in S_9$ so that \begin{align*}
    v_{(2\,3)(4\,5).T}=(4\,8)(3\,6)(5\,9)v_T.
\end{align*} This observation will be very helpful in developing an analogous straightening algorithm for semistandard multiset partition tableaux.
\end{remark}

\subsection{Super RSK}

A \defn{superalphabet} is a totally ordered set $\mathscr{A}$ with a map $\mo:\mathscr{A}\to \Z/2\Z=\{0,1\}$. Recall that in $\Z/2\Z$, we have that $1+1=0$. If $\mo(a)=0$ we say that $a$ is \defn{even}, and if $\mo(a)=1$ we say that $a$ is \defn{odd}. For superalphabets $\mathscr{A}$ and $\mathscr{B}$, we call an element of $\mathscr{A}\times \mathscr{B}$ a \defn{biletter} $(a,b)$. We define a total order\footnote{In \cite{Muth2019Super}, the author denotes this order with $a_i$ and $b_i$ swapped as $\lhd$ and the order on the second coordinate given below as $\prec$.} on biletters where $(a_1,b_1)\prec (a_2,b_2)$ if either

\begin{enumerate}
    \item $a_1<a_2$
    \item $a_1=a_2$ and one of the following hold: \begin{enumerate}
        \item $\mo(b_1)=1$ and $\mo(b_2)=0$.
        \item $\mo(b_1)=\mo(b_2)=0$ and $b_1<b_2$.
        \item $\mo(b_1)=\mo(b_2)=1$ and $b_1>b_2$.
    \end{enumerate}
\end{enumerate}

A biletter is \defn{mixed} if $\mo(a)+\mo(b)=1$. A \defn{biword} of length $m$ is a sequence $w=((a_1,b_1),\dots,(a_m,b_m))\in(\mathscr{A}\times\mathscr{B})^{m}$. A biword is called \defn{ordered} if $(a_i,b_i)\preceq (a_j,b_j)$ whenever $i<j$ and called \defn{restricted} if no mixed biletter is repeated. For $\tilde A$ a multiset of elements in $\mathscr{A}$ and $\tilde B$ a multiset of elements in $\mathscr{B}$ with $\abs{\tilde A}=\abs{\tilde B}=m$, write $\ORBW(\tilde A,\tilde B)$ for the set of ordered restricted biwords $((a_1,b_1),\dots,(a_m,b_m))$ where $\multi{a_1,\dots,a_m}=\tilde A$ and $\multi{b_1,\dots,b_m}=\tilde B$.

Let $\tilde A$ be a multiset from a superalphabet $\mathscr{A}$. A \defn{semistandard supertableau} of shape $\lambda$ with content $\tilde A$ is a filling $T$ of the Young diagram of $\lambda$ with elements of $A$ so that
\begin{enumerate}
    \item The rows and columns of $T$ are non-decreasing.
    \item $T$ is row-strict in the odd letters: if $a$ appears twice in the same row, then $\mo(a)=0$.
    \item $T$ is column-strict in the even letters: if $a$ appears twice in the same column, then $\mo(a)=1$.
\end{enumerate} Write $\SSsT(\lambda,\tilde A)$ for the set of semistandard supertableaux of shape $\lambda$ with content $\tilde A$.

Following \cite{Muth2019Super}, we now define two operations on the set of semistandard supertableaux called $0$-insertion and $1$-insertion.

\begin{definition}\label{def:zero_insertion} The process of performing $0$-insertion of a value $a$ into a tableau depends on whether $a$ is odd or even. We describe the process compactly here by writing the instructions for an even value first in \textcolor{c2}{orange} and then the instructions for an odd value in \textcolor{c1}{blue}.

Suppose $a$ is (\textcolor{c2}{even}/\textcolor{c1}{odd}). To perform $0$-insertion of $a$ into a (\textcolor{c2}{row}/\textcolor{c1}{column}) of a tableau $T$,
\begin{itemize}
    \item[(i)] If $a$ is larger than or equal to each element in the (\textcolor{c2}{row}/\textcolor{c1}{column}), place it at the end of the (\textcolor{c2}{row}/\textcolor{c1}{column}) and terminate.
    \item[(ii)] Otherwise, consider the (\textcolor{c2}{left}/\textcolor{c1}{bottom})-most box in the (\textcolor{c2}{row}/\textcolor{c1}{column}) with content larger than $a$ and call it the bump site. Write $b$ for the content of the bump site, and replace the content of the bump site with $a$.
    \item[(iii)] If $b$ is even, insert it into the row above the bump site. If $b$ is odd, insert it into the column to the right of the bump site.
\end{itemize}

Write $T\stackrel{0}{\leftarrow}a$ for the result of inserting $a$ into the first (\textcolor{c2}{row}/\textcolor{c1}{column}) of $T$.
\end{definition}

\begin{example} We give an example of $0$-insertion where the alphabet is the positive integers with the usual map into $\Z/2\Z$. We color the odd numbers blue and even numbers orange to highlight which rules are being applied from \Cref{def:zero_insertion}\begin{center}
    \input{figures.ex_zero_insertion}
\end{center}
\end{example}

Similarly, define $1$-insertion as in \Cref{def:zero_insertion}, swapping ``row'' and ``column,'' modifying step (i) to check if $a$ is \textit{strictly} larger, and modifying step (ii) to look for a box with content larger than or \textit{equal} to $a$.

\begin{example} We repeat the previous example, inserting $1$ via $1$-insertion instead.
    \input{figures.ex_one_insertion}
\end{example}

These dual notions of $0$- and $1$-insertion give rise to the following algorithm called \defn{super RSK}.

\begin{definition}
    For $w=((a_1,b_1),\dots,(a_m,b_m))\in (\mathscr{A}\times\mathscr{B})^m$ an ordered, restricted biword, define $P_w^0=\emptyset$ and $Q_w^0=\emptyset$. Then inductively define \[P_w^i=P_w^{i-1}\stackrel{\mo(a_i)}{\leftarrow} b_i\] and $Q_w^i$ to be $Q_w^{i-1}$ along with $a_i$ in the box created in the insertion to $P_w^{i-1}$. Write $P_w=P_w^m$ and $Q_w=Q_w^m$ and\begin{align*}
        \sRSK(w)=(P_w,Q_w).
    \end{align*}
\end{definition}

\begin{theorem}\cite[Theorem 5.2]{Muth2019Super} For any multiset $\tilde A$ from $\mathscr{A}$ and multiset $\tilde B$ from $\mathscr{B}$ with $\abs{\tilde A}=\abs{\tilde B}=m$, the map $\sRSK$ defines a bijection \begin{align*}
    \mathcal{ORBW}(\tilde A,\tilde B,m)\stackrel{\sim}{\longleftrightarrow}\biguplus_{\lambda\vdash m} \SSsT(\lambda,\tilde A)\times\SSsT(\lambda,\tilde B)
\end{align*}
\end{theorem}

\section{Super Multiset Partition Insertion}\label{sec:emp_insertion}

In this section, we present a modified version of the insertion algorithm on multiset partitions introduced in \cite{Colmenarejo2020Insertion} built on the super RSK algorithm. Let $\mathscr{M}$ be the set of multisets with elements in $\{1,2,\dots\}\cup\{\ov1,\ov2,\dots\}$ ordered by the last-letter order. This is a superalphabet with grading given by the parity of the number of barred elements:

\begin{align*}
    \mo(\tilde{S})&=\begin{cases}
        0 & \text{$\tS$ has an even number of barred elements}\\
        1 & \text{$\tS$ has an odd number of barred elements}
    \end{cases}.
\end{align*}

\begin{definition}
    Let $\hpi\in\hat\Pi_{2(\bfa,\bfb)}$. Let $\multi{\tS_1,\dots,\tS_m}$ be the multiset of propagating blocks of of $\hpi$. Write $\tS_i^-$ for the multiset of underlined values in $\tS_i$ with the underlines removed and write $\tS_i^+$ for the remaining elements. Define the \defn{biword associated to $\hpi$} to be the ordered biword \[w(\hpi)=\left(\begin{tabular}{ccc}
         $\tS_1^+$ & $\cdots$ & $\tS_m^+$ \\
         $\tS_1^-$ & $\cdots$ & $\tS_m^-$
    \end{tabular}\right)\in (\mathscr{M}\times\mathscr{M})^m.\] By definition of the set $\hat\Pi_{2(\bfa,\bfb)}$, the biword $w(\tpi)$ is restricted.
\end{definition}

\begin{definition} Let $\hat\pi\in\hat\Pi_{2(\bfa,\bfb)}$ and $n\geq 2\abs{\bfa}+2\abs{\bfb}$. Then $\sRSK(w(\hat\pi))=(P,Q)$ is a pair of semistandard supertableaux of the same shape $\mu$ with content equal to the restrictions of the propagating blocks of $\hat\pi$ to the underlined and non-underlined elements respectively. Note that $n$ is at least twice the number of propagating blocks of $\hat\pi$, so in particular $n-\abs{\mu}\geq\abs{\mu}$. Hence, we can form a tableau $P'$ of shape $\lambda=(n-\abs{\mu},\mu)$ from $P$ by adding a first row consisting of a number of empty boxes followed by the blocks of $\hat\pi$ with only non-underlined elements. Do the same for $Q'$ adding the blocks with only underlined elements. Let \[\smRSK(\hpi)=(P',Q').\]
\end{definition}

\begin{prop}
    For $n\geq2\abs\bfa+2\abs\bfb$, the map $\smRSK$ is a bijection:
    \begin{align}
        \hat\Pi_{2(\bfa,\bfb)}\stackrel{\sim}{\longleftrightarrow}\biguplus_{\substack{\lambda\vdash n\\\SSMT(\lambda,\bfa,\bfb)\neq\emptyset}} \SSMT(\lambda,\bfa,\bfb)\times\SSMT(\lambda,\bfa,\bfb)\label{eq:smRSK_bijection}
    \end{align}
\end{prop}

\begin{proof}
    Note that $\hpi\in\hat\Pi_{2(\bfa,\bfb)}$ cannot have two identical non-propagating blocks with an odd number of barred elements. From this observation and the definition of semistandard supertableaux, it's clear that the image of $\hpi$ under $\smRSK$ is a pair of semistandard multiset partition tableaux in $\SSMT(\lambda,\bfa,\bfb)$.

    Given such a pair, we can recover the multiset partition $\hpi$ by first removing the boxes in the first row to obtain the non-propagating blocks and then reversing super RSK on the remaining tableaux to form the propagating blocks.
\end{proof}

\begin{example}We give the result of $\smRSK$ on a restricted multiset partition $\hat\pi$.
    \input{figures.ex_multiset_insertion}
\end{example}

We will see in the next two sections that this bijection is an enumerative manifestation of the decomposition of an algebra $\MP_{\bfa,\bfb}(n)$ as a bimodule over itself.

\section{The mixed multiset partition algebra \texorpdfstring{$\MP_{\bfa,\bfb}(x)$}{MPab(x)}}\label{sec:epa}

Let $\bfa\in W_{r,k}$ and $\bfb\in W_{s,m}$. The subgroup $S_\bfa\times S_\bfb\subseteq S_{r+s}$ consists of the permutations of $[r+s]$ which respect the coloring map $\kappa_{\bfa,\bfb}$. Explicitly, \begin{align*}
    S_\bfa\times S_\bfb&=\left(S_{\kappa_{\bfa,\bfb}^{-1}(1)}\times\dots\times S_{\kappa_{\bfa,\bfb}^{-1}(k)}\right)\times\left(S_{\kappa_{\bfa,\bfb}^{-1}(\ov1)}\times\dots\times S_{\kappa_{\bfa,\bfb}^{-1}(\ov m)}\right).
\end{align*}

\begin{example}
    \begin{align*}
        S_{(3,2)}\times S_{(2,1,1)}=S_{\{1,2,3\}}\times S_{\{4,5\}}\times S_{\{6,7\}}\times S_{\{8\}}\times S_{\{9\}}.
    \end{align*}
\end{example}

We will usually write elements of $S_\bfa\times S_\bfb$ as $\sigma\tau$ with $\sigma\in S_\bfa$ and $\tau\in S_\bfb$.

Let $e_{\bfa,\bfb}$ be the idempotent given by \begin{align*}
    e_{\bfa,\bfb}&=\frac{1}{\abs{S_\bfa\times S_\bfb}}\sum_{\sigma\tau\in S_\bfa\times S_\bfb}\sgn(\tau)\sigma\tau\in \C S_{r+s}\subseteq P_{r+s}(x).
\end{align*} Define the \defn{mixed multiset partition algebra} by \[\MP_{\bfa,\bfb}(x)\defeq e_{\bfa,\bfb} P_{r+s}(x)e_{\bfa,\bfb}.\] We begin this section by providing a basis for $\MP_{\bfa,\bfb}(x)$ indexed by elements of $\hat\Pi_{2(\bfa,\bfb)}$.

\begin{lemma}\label{lem:sum_to_zero} Let $G$ be a finite group which acts on a finite set $S$, and let $f:G\to\{\pm1\}$ be a group homomorphism. let $\C S$ be the formal $\C$-span of the elements of $S$. For $s\in S$, \begin{align}\label{eq:sum_to_zero}
    \sum_{g\in G}f(g) g.s=0
\end{align} if and only if $s$ is fixed by an element $g\in G$ with $f(g)=-1$.
\end{lemma}

\begin{proof}
    First, suppose \Cref{eq:sum_to_zero} holds. Then in particular, the coefficient on $s$ in the sum is zero. Because $s$ occurs with positive coefficient (when $g$ is the identity element), it must occur with a negative coefficient as well. Hence, there exists $g\in G$ with $g.s=s$ and $f(g)=-1$.

    Conversely, suppose there exists $g\in G$ with $g.s=s$ and $f(g)=-1$. Let $H=\langle g\rangle\subseteq G$ be the subgroup generated by $g$. Let $t_1,\dots,t_k$ be a transversal of $H$ in $G$ so that $G=\uplus_{i=1}^k t_i H$. Then \begin{align*}
        \sum_{g\in G}f(g) g.s&=\sum_{i=1}^k \sum_{j=1}^{\abs{H}} f(t_ig^j)t_ig^j.s\\
        &=\sum_{i=1}^k f(t_i) \left(\sum_{j=1}^{\abs{H}} f(g^j)\right)t_i.s\\
        &=\sum_{i=1}^k f(t_i) \left(\sum_{j=1}^{\abs{H}} (-1)^j\right)t_i.s.
    \end{align*} Because $f(g)=-1$, it must be that $\abs{H}$ is even, so $\sum_{j=1}^{\abs{H}}(-1)^j=0$ and hence \Cref{eq:sum_to_zero} holds.
\end{proof}

\begin{lemma}\label{lem:projection_nonzero} For $\pi\in \Pi_{2(r+s)}$, \[e_{\bfa,\bfb}\LL_\pi e_{\bfa,\bfb}=0\] if and only if $\kappa_{\bfa,\bfb}(\pi)\notin \hat\Pi_{2(\bfa,\bfb)}$.
\end{lemma}

\begin{proof}
    For ease of notation, let $G=(S_\bfa\times S_\bfb)^2$. Applying \Cref{lem:sum_to_zero} with $f:G\to\{\pm1\}$ given by $f(\sigma_1\tau_1,\sigma_2\tau_2)=\sgn(\tau_1)\sgn(\tau_2)$ gives us that $e_{\bfa,\bfb}\LL_\pi e_{\bfa,\bfb}=0$ if and only if $\pi$ is fixed by some element \[(\sigma_1\tau_1,\sigma_2\tau_2)\in(S_\bfa\times S_\bfb)^2\] with $\sgn(\tau_1)\sgn(\tau_2)=-1$.

    Let $G^\pi$ be the subgroup of $G$ which fixes $\pi$. By \Cref{lem:fixing_partition}, \begin{align*}
        G^\pi=\left(\prod_{B\in\pi} N_B\cap G\right)\rtimes\left(H\cap G\right).
    \end{align*} Then there exists $g\in G^\pi$ with $f(g)=-1$ if and only if one or both of the following hold:

    \begin{enumerate}
        \item[(i)] There exists $g\in N_B\cap G$ for some $B\in\pi$ with $f(g)=-1$.
        \item[(ii)] There exists $g\in H\cap G$ with $f(g)=-1$.
    \end{enumerate}

    In the first case, an element of $N_B\cap G$ is a permutation of elements within the block $B$ which preserves what color the element is mapped to under $\kappa_{\bfa,\bfb}$. There exists $g\in N_B\cap G$ with $f(g)=-1$ if and only if there are at least two elements of $B$ which map to the same barred element under $\kappa_{\bfa,\bfb}$.

    In the latter case, an element of $H\cap G$ is a permutation of the blocks of $\pi$ which only sends a block $B$ to another block mapped to the same multiset under $\kappa_{\bfa,\bfb}$. There exists $g\in H\cap G$ with $f(g)=-1$ if and only if there is a repeated block of $\kappa_{\bfa,\bfb}(\pi)$ with an odd number of barred elements.

    Hence, $e_{\bfa,\bfb}\LL_\pi e_{\bfa,\bfb}=0$ if and only if $\kappa_{\bfa,\bfb}(\pi)\notin\hat\Pi_{2(\bfa,\bfb)}$.
\end{proof}

\begin{definition} Let $\hat\pi\in\hat\Pi_{2(\bfa,\bfb)}$ and let $\pi$ be its standardization. Define \begin{align*}
    \DD_\hpi&\defeq e_{\bfa,\bfb} \LL_\pi e_{\bfa,\bfb}.
\end{align*}
\end{definition}

\begin{theorem} The set $\{\DD_\hpi:\hpi\in\hat\Pi_{2(\bfa,\bfb)}\}$ forms a basis for $\MP_{\bfa,\bfb}(x)$.\end{theorem}
\begin{proof}
    By \Cref{lem:projection_nonzero}, $e_{\bfa,\bfb}\LL_\pi e_{\bfa,\bfb}=0$ unless $\kappa_{\bfa,\bfb}(\pi)\in\hat\Pi_{2(\bfa,\bfb)}$. Hence, the set \[\{e_{\bfa,\bfb}\LL_\pi e_{\bfa,\bfb}:\pi\in\Pi_{2(r+s)},\kappa_{\bfa,\bfb}(\pi)\in\hat\Pi_{2(\bfa,\bfb)}\}\] spans $\MP_{\bfa,\bfb}(x)$. Furthermore, if $\kappa_{\bfa,\bfb}(\pi)=\kappa_{\bfa,\bfb}(\nu)$, then $\pi=(\sigma_1\tau_1,\sigma_2\tau_2).\nu$ for some $(\sigma_1\tau_1,\sigma_2\tau_2)\in (S_\bfa\times S_\bfb)^2$, so $e_{\bfa,\bfb}\LL_\pi e_{\bfa,\bfb}=\sgn(\tau_1)\sgn(\tau_2)e_{\bfa,\bfb}\LL_\nu e_{\bfa,\bfb}$. We then need only one representative from each preimage $\kappa_{\bfa,\bfb}^{-1}(\hat\pi)$, and so the set \[\{\DD_\hpi:\hpi\in\hat\Pi_{2(\bfa,\bfb)}\}\] spans $\MP_{\bfa,\bfb}(x)$. Because the elements $\{\DD_\hpi:\hpi\in\hat\Pi_{2(\bfa,\bfb)}\}$ are nonzero sums over distinct $(S_\bfa\times S_\bfb)^2$-orbits, they are linearly independent and hence form a basis of $\MP_{\bfa,\bfb}(x)$.
\end{proof}

\begin{remark} This description of the basis leads immediately to a product formula. Given $\hpi,\hnu\in\hat\Pi_{2\bfa}$ and $\pi,\nu$ their standardizations, \begin{align*}
    \DD_\hpi\DD_\hnu&=(e_{\bfa,\bfb}\LL_\pi e_{\bfa,\bfb})(e_{\bfa,\bfb}\LL_\nu e_{\bfa,\bfb})\\
    &=\frac{1}{\abs{S_\bfa\times S_\bfb}}\sum_{\sigma\tau\in S_\bfa\times S_\bfb}\sgn(\tau) e_\bfa \LL_\pi \sigma \tau \LL_\nu e_\bfa\\
    &=\frac{1}{\abs{S_\bfa\times S_\bfb}}\sum_{\sigma\tau\in S_\bfa\times S_\bfb}\sgn(\tau) x^{c(\pi,\sigma\tau.\nu)}e_\bfa \LL_{\pi\circ(\sigma\tau.\nu)} e_\bfa. 
\end{align*} On diagrams, this amounts to the following process. First draw the diagrams for $\hpi$ and $\hnu$ in standard form and place the former on top of the latter. Then average over all permutations of identically colored vertices on top of the diagram for $\hnu$, taking their product as set partition diagrams and recording the sign of the permutation of the open circles.
\end{remark}

\begin{example} 
    \input{figures.ex_epa_product}
\end{example}

\section{Representations of \texorpdfstring{$\MP_{\bfa,\bfb}(n)$}{MPab(n)}}\label{sec:epa_reps}

In order to construct the simple $\MP_{\bfa,\bfb}(n)$-modules, we make the following definition inspired by \Cref{prop:projected_irreps}.

\begin{definition} Let $\bfa\in W_{r,k}$, $\bfb\in W_{s,m}$ and $n\geq2\abs\bfa+2\abs\bfb$. For $\lambda\in\Lambda^{P_{r+s}(n)}$, define \begin{align*}
    \MP_{\bfa,\bfb}^\lambda\defeq e_{\bfa,\bfb} P_{r+s}^\lambda.
\end{align*}
\end{definition}

Let $\Lambda^{\MP_{\bfa,\bfb}(n)}$ be the set of $\lambda\in\Lambda^{P_{r+s}(n)}$ for which $\MP_{\bfa,\bfb}^\lambda\neq0$. By \Cref{prop:projected_irreps} and the double centralizer theorem, we know that \[\{\MP_{\bfa,\bfb}^\lambda:\lambda\in\Lambda^{\MP_{\bfa,\bfb}(n)}\}\] forms a complete set of simple $\MP_{\bfa,\bfb}(n)$-modules. This section is devoted to providing bases for these modules.

\begin{definition} For $\tT\in\MT(\lambda,\bfa,\bfb)$, let $T$ be its standardization and define \[y_\tT\defeq e_{\bfa,\bfb} v_T.\] For $S\in\kappa_{\bfa,\bfb}^{-1}(\tT)$, let $\sigma_S\tau_S\in S_\bfa\times S_\bfb$ be any permutation so that $\sigma_S\tau_S v_S=v_T$. Then \begin{align*}
    e_{\bfa,\bfb} v_S&=\sgn(\tau_S)y_{\tT}.
\end{align*}
\end{definition}

\begin{lemma}\label{lem:no_repeated_even}
    Recall that an even multiset is one with an even number of barred elements. If $\tT\in\MT(\lambda,\bfa,\bfb)$ has a repeated even entry in a column, then $y_{\tT}=0$.
\end{lemma}

\begin{proof}
    If $\tT$ has a repeated even entry in a column, then any $T\in\kappa_{\bfa,\bfb}^{-1}(\tT)$ has two boxes with contents $R$ and $S$ in the same column with $\kappa_{\bfa,\bfb}(R)=\kappa_{\bfa,\bfb}(S)$ having an even number of barred elements. For any permutation $\sigma\tau\in S_\bfa\times S_\bfb$ with $\sigma\tau(R)=S$ and $\sigma\tau(S)=R$ which fixes every other element of $[r]$, $\tau$ must then be an even permutation. Furthermore, $\sigma\tau v_T=v_{T'}$ where $T'$ is the tableau obtained by swapping the two boxes in the same column. By the straightening algorithm for set partition tableaux, we have that $\sigma\tau v_T=-v_T$. Applying $e_{\bfa,\bfb}$ to both sides yields the following:

    \begin{align*}
        e_{\bfa,\bfb} \sigma\tau v_T&=-e_{\bfa,\bfb} v_T\\
        \sgn(\tau) e_{\bfa,\bfb} v_T&=-e_{\bfa,\bfb} v_T\\
        e_{\bfa,\bfb} v_T&=-e_{\bfa,\bfb} v_T,
    \end{align*} so $e_{\bfa,\bfb} v_T=0$ and hence $y_\tT=0$.
\end{proof}

We may then assume that our tableaux $\tT$ do not have repeated even entries in their columns.

\begin{lemma}\label{lem:exterior_tableau_straightening}
    Let $\tT\in\MT(\lambda,\bfa,\bfb)$ and suppose $\tT$ has a decrease (or a weak decrease of odd entries) along a row. Then \begin{align*}
        y_\tT&=\sum_{\tS}c_\tS y_\tS
    \end{align*} where the sum is over $\tS$ where the decrease (or weak decrease) has been eliminated and $\tS$ is larger than $\tT$ in the column dominance order.
\end{lemma}

\begin{proof}
    Suppose $\tT$ has a decrease (or weak decrease of odd entries) in a row. Then there is an $S\in\kappa_{\bfa,\bfb}^{-1}(\tT)$ which has the same content as the standardization of $\tT$ but has a decrease $a>b$ in the same position as the (weak) decrease in $\tT$. Let $A$ be the set of boxes of $S$ including $a$ and above and $B$ be the set of boxes including $b$ and below. Let $g(A,B)$ be the Garnir transversal associated to these subsets. Then, \begin{align*}
        \sum_{\eta\in g(A,B)}\sgn(\eta)v_{\eta.S}=0
    \end{align*} where each $\eta.S$ has the decrease removed and is larger than $S$ in the column-dominance order. Then multiplying by $e_{\bfa,\bfb}$, we get \begin{align*}
        \sum_{\eta\in g(A,B)}\sgn(\eta)e_{\bfa,\bfb} v_{\eta.S}&=\sum_{\eta\in g(A,B)}\sgn(\eta)\sgn(\tau_{\eta.S})y_{\kappa_{\bfa,\bfb}(\eta.S)}=0
    \end{align*} where $\sigma_{\eta.S}\tau_{\eta.S}\in S_\bfa\times S_\bfb$ is any element so that $\sigma_{\eta.S}\tau_{\eta.S} v_{\eta.S}=v_T$ where $T$ is the standardization of $\kappa_{\bfa,\bfb}(\eta.S)$.

    Each $\kappa_{\bfa,\bfb}(\eta.S)\neq\tT$ has the decrease (or weak decrease) eliminated and by \Cref{rmk:col_dominance} is larger than $\tT$ in the column dominance order. We then need only show that the coefficient on $y_\tT$ in the above sum is nonzero.

    If $\kappa_{\bfa,\bfb}(\eta.S)=\tT$, then $\eta$ only permutes boxes of $S$ which are identical under $\kappa_{\bfa,\bfb}$. That is, there exists a permutation $\sigma\tau\in S_\bfa\times S_\bfb$ such that $v_{\eta.S}=\sigma\tau v_S$ (see \Cref{rmk:straightening}). We now compute an element $\sigma_{\eta.S}\tau_{\eta.S}$ such that $\sigma_{\eta.S}\tau_{\eta.S}v_{\eta.S}=v_T$ starting from such an element for $S$.

    \begin{align*}
        \sigma_S\tau_S v_S&=v_T\\
        \sigma_S\tau_S(\sigma\tau)^{-1}\sigma\tau v_S&=v_T\\
        (\sigma_S\sigma^{-1})(\tau_S\tau^{-1})v_{\eta.S}&=v_T
    \end{align*}

    Hence, we may choose $\sigma_{\eta.S}\tau_{\eta.S}=(\sigma_S\sigma^{-1})(\tau_S\tau^{-1})$. Then, \begin{align*}
        \sgn(\eta)\sgn(\tau_{\eta.S})&=\sgn(\eta)\sgn(\tau_S \tau^{-1})\\
        &=\sgn(\eta)\sgn(\tau)\sgn(\tau_S).
    \end{align*}

    Note that because we can assume that $A$ and $B$ are strictly increasing in even entries and we are not considering the case of a weak even decrease, $A\cup B$ has no repeated even entries. Hence, $\eta$ permutes only odd boxes, leaving the even boxes fixed, so $\sgn(\eta)=\sgn(\tau)$. Then,

    \begin{align*}
        \sgn(\eta)\sgn(\tau_{\eta.S})&=\sgn(\tau_S).
    \end{align*}

    Because each $y_\tT$ in the sum appears with the same nonzero coefficient $\sgn(\tau_S)$, the coefficient on $y_\tT$ is nonzero, and we have written $y_\tT$ as the desired linear combination.
\end{proof}

\begin{theorem}\label{thm:basis} The set $\{y_\hT:\hT\in\SSMT(\lambda,\bfa,\bfb)\}$ forms a basis for $\MP_{\bfa,\bfb}^\lambda$.
\end{theorem}

\begin{proof}

    Clearly, the set $\{e_{\bfa,\bfb} v_T:T\in\SMT(\lambda,r+s)\}$ spans $\MP_{\bfa,\bfb}^\lambda$, and so by \Cref{lem:no_repeated_even} the vectors $y_\tT$ for $\tT\in\MT(\lambda,\bfa,\bfb)$ without a repeated even entry in any column span $\MP_{\bfa,\bfb}^\lambda$. If $\tT$ has a repeated odd-sized entry within a row, we can apply \Cref{lem:exterior_tableau_straightening} to write it as a linear combination of multiset partition tableaux which do not have this repeat. The resulting tableaux may have repeated odd entries or decreases in other locations, in which case we repeatedly apply the lemma. Because the tableaux are always larger in the column dominance order, this process eventually terminates, writing $y_\tT$ as a linear combination of $y_\hS$ where $\hS\in\SSMT(\lambda,\bfa,\bfb)$. Hence, the set $\{y_\hT:\hT\in\SSMT(\lambda,\bfa,\bfb)\}$ spans $\MP_{\bfa,\bfb}^\lambda$.

    Comparing dimensions in the decomposition of $\MP_{\bfa,\bfb}(n)$ as a bimodule over itself and applying the super multiset RSK bijection (\ref{eq:smRSK_bijection}), we obtain the following equation.

\begin{align*}
    \sum_{\lambda\in\Lambda^{\MP_{\bfa,\bfb}(n)}} (\dim(\MP_{\bfa,\bfb}^\lambda))^2&=\dim(\MP_{\bfa,\bfb}(n))\\
                                                                    &=\abs{\hat\Pi_{2(\bfa,\bfb)}}\\
                                                                     &=\sum_{\lambda\in\Lambda^{\MP_{\bfa,\bfb}(n)}} \abs{\SSMT(\lambda,\bfa,\bfb)}^2
\end{align*}

The spanning set above shows that $\dim(\MP_{\bfa,\bfb}^\lambda)\leq\abs{\SSMT(\lambda,\bfa,\bfb)}$ for each $\lambda$, so the above equality implies that $\dim(\MP_{\bfa,\bfb}^\lambda)=\abs{\SSMT(\lambda,\bfa,\bfb)}$ and the spanning set $\{y_\hT:\hT\in\SSMT(\lambda,\bfa,\bfb)\}$ is in fact a basis.
\end{proof}

The construction of the simple $\MP_{\bfa,\bfb}(n)$-modules as projections of simple $P_{r+s}(n)$-modules leads to a straightforward formula for the action of a multiset partition on a semistandard multiset partition tableau. That is, for $\kappa_{\bfa,\bfb}(\pi)=\tpi$,

\begin{align*}
    e_{\bfa,\bfb}\LL_\pi e_{\bfa,\bfb} . e_{\bfa,\bfb} v_T&=\sum_{(\sigma,\tau)\in S_\bfa\times S_\bfb} \sgn(\tau) e_{\bfa,\bfb} (\LL_{\pi.\sigma\tau}.v_T).
\end{align*} In terms of diagrams, this formula amounts to first pulling out the content of the tableau as a one-row diagram, placing the diagram for $\tpi$ on top. Then average over all permutations of identically-colored vertices at the bottom of $\tpi$, noting the sign of the permutation of the open circles. The action is then computed exactly as for the partition algebra and the results summed with the appropriate sign.

\begin{example} Here we show the action of a multiset partition on a tableau.
    \input{figures.ex_mp_action.tex}
\end{example}

\section{Decompositions of Multivariate Polynomial Rings}\label{sec:conclusions}

In this final section, we apply the above results on super RSK and the mixed multiset partition algebra to describe how a certain multivariate polynomial ring decomposes as an $S_n$-module and $GL_n$-module. Let \[\C[X_{n\times m};\Theta_{n\times m'}]\defeq \C[x_{i,j},\theta_{i,j}:1\leq i\leq n,1\leq j\leq m,1\leq j'\leq m']\] where the variables $x_{i,j}$ commute with all variables and $\theta_{i,j}\theta_{a,b}=-\theta_{a,b}\theta_{i,j}$ if either $i\neq a$ or $j\neq b$ and ${\theta_{i,j}}^2=0$. For $r+s=n$, $\bfa\in W_{r,m}$ and $\bfb\in W_{s,m'}$, write \[\C[X_{n\times m};\Theta_{n\times m'}]^{\bfa,\bfb}\] for the subspace spanned by monomials with degree $\bfa_t$ in the variables $x_{i,t}$ and degree $\bfb_t$ in the variables $\theta_{i,t}$. 

\begin{example}
    \begin{align*}
        \textcolor{c1}{x_{11}x_{11}x_{31}}\textcolor{c2}{x_{22}x_{32}}\textcolor{c3}{\theta_{11}\theta_{21}}\textcolor{c4}{\theta_{42}}\theta_{23}\in\C[X_{3\times 5};\Theta_{3\times4}]^{(\textcolor{c1}{3},\textcolor{c2}{2}),(\textcolor{c3}{2},\textcolor{c4}{1},1)}.
    \end{align*}
\end{example}

\subsection{As an $S_n$-module}

To understand the ring $\C[X_{n\times m};\Theta_{n\times m'}]$ as an $S_n$-module, we want to think of it as decomposing into projections by the idempotents $e_{\bfa,\bfb}$. That is, we notice that \begin{align*}
    \C[X_{n\times m};\Theta_{n\times m'}]^{\bfa,\bfb}&\cong e_{\bfa,\bfb}{V_n}^{\otimes n}.
\end{align*} The following theorem can then be recovered as a straightforward application of \Cref{thm:dct} and \Cref{thm:basis}.

\begin{theorem}\label{thm:restriction}\cite[Theorem 3.1]{orellana2020combinatorial} Let $n\geq2\abs{\bfa}+2\abs{\bfb}$. Then as an $S_n$-module, \begin{align*}
    \C[X_{n\times m};\Theta_{n\times m'}]^{\bfa,\bfb}&\cong\bigoplus_{\lambda\in\Lambda^{\MP_{\bfa,\bfb}(n)}}\left(S^\lambda\right)^{\oplus\abs{\SSMT(\lambda,\bfa,\bfb)}}.
\end{align*}
\end{theorem}

\subsection{As a $GL_n$-module}\label{sec:GLn_module} In this section we employ the super RSK correspondence to provide a combinatorial interpretation of the decomposition of $\C[X_{n\times m};\Theta_{n\times m'}]$ as a $GL_n$-module. To that end, we want to view the ring as decomposing into tensor products of symmetric and alternating powers. That is, we notice that \begin{align*}
    \C[X_{n\times m};\Theta_{n\times m'}]^{\bfa,\bfb}&\cong \Sym^\bfa(\C^n)\otimes \bigwedge\nolimits^\bfb(\C^n)\intertext{where}
    \Sym^\bfa(\C^n)&=\Sym^{\bfa_1}(\C^n)\otimes\dots\otimes \Sym^{\bfa_m}(\C^n)
    \intertext{and}
    \bigwedge\nolimits^\bfb(\C^n)&=\bigwedge\nolimits^{\bfb_1}(\C^n)\otimes\dots\otimes\bigwedge\nolimits^{\bfb_{m'}}(\C^n).
\end{align*} From this perspective, it's clear that the character of this subspace as a $GL_n$ module is given by $h_\bfa(X_n)e_\bfb(X_n)$ where \[h_\bfa=h_{\bfa_1}\cdots h_{\bfa_m} \text{ and } e_\bfb=e_{\bfb_1}\cdots e_{\bfb_{m'}}.\]

Let $\SSMT'(\lambda,\bfa,\bfb)$ be the set of semistandard multiset partition tableaux whose content has all blocks of size one and whose unbarred and barred elements have multiplicities given by $\bfa$ and $\bfb$ respectively.

\begin{theorem}\label{thm:he_expansion} Let $\bfa\in W_{r,m}$ and $\bfb\in W_{s,m'}$. Then,
    \begin{align*}
        h_\bfa e_\bfb = \sum_{\lambda\vdash r+s} \abs{\SSMT'(\lambda, \bfa, \bfb)}s_\lambda
    \end{align*}
\end{theorem}

\begin{proof}
    Each monomial in $h_\bfa e_\bfb=h_{\bfa_1}\dots h_{\bfa_m}e_{\bfb_1}\dots e_{\bfb_{m'}}$ is of the form \begin{align*}
        (x_{i_1^{(1)}}\dots x_{i_{\bfa_1}^{(1)}})\dots(x_{i_1^{(m)}}\dots x_{i_{\bfa_m}^{(m)}})(x_{j_1^{(1)}}\dots x_{j_{\bfb_1}^{(1)}})\dots(x_{j_1^{(m')}}\dots x_{j_{\bfb_{m'}}^{(m')}})
    \end{align*} where each factor $(x_{i_1^{(s)}}\dots x_{i_{\bfa_s}^{(s)}})$ is the contribution of $h_{\bfa_s}$ and $(x_{j_1^{(t)}}\dots x_{j_{\bfb_t}^{(t)}})$ is the contribution of $e_{\bfb_t}$. These monomials are in bijection with ordered restricted biwords where the monomial as written above corresponds to the biword obtained by putting the multiset \begin{align*}
        \multi{(s,i_1^{(s)}),\dots,(s,i_{\bfa_s}^{(s)}):1\leq s\leq m}\uplus\multi{(\ov{t},j_1^{(t)}),\dots,(\ov{t},j_{\bfb_t}^{(t)}):1\leq t\leq m'}
    \end{align*} in order.

    Under the super RSK correspondence, these biwords are in bijection with pairs $(U,T)$ of the same shape $\lambda\vdash r+s$ where \begin{enumerate}
        \item $U$ is a semistandard Young tableau whose content is precisely the multiplicities of the subscripts of the variables in the monomial.
        \item $T$ is a semistandard multiset partition tableau with entries all size one and multiplicities of unbarred and barred values given by $\bfa$ and $\bfb$ respectively.
    \end{enumerate} Rearranging by partition shape, we can rewrite the sum \begin{align*}
        h_\bfa e_\bfb&=\sum_{\lambda\vdash r+s} \sum_{(U,T)} x^U\\
        &=\sum_{\lambda\vdash r+s} \abs{\SSMT'(\lambda,\bfa,\bfb)}\sum_{U\in\SSYT(\lambda)} x^U\\
        &=\sum_{\lambda\vdash r+s} \abs{\SSMT'(\lambda,\bfa,\bfb)}s_\lambda.
    \end{align*}
\end{proof}

\begin{example} Here we show the correspondence for a monomial in $h_3h_2e_2e_2$.
    \[(x_1x_1x_3)(x_2x_3)(x_1x_2)(x_2x_4)\]
    \[\updownarrow\]
    \[\left(\begin{array}{ccccccccc}
            1 & 1 & 1 & 2 & 2 & \ov1 & \ov1 & \ov2 & \ov2\\
            1 & 1 & 3 & 2 & 3 & 1 & 2 & 2 & 4
        \end{array}\right)\]
        \[\updownarrow\]
        \[\left(\inline{\young(11123,223,4)},\inline{\young(1112<\ov1>,2<\ov1><\ov2>,<\ov2>)}\right)\]
\end{example}

\begin{corollary} As a $GL_n$-module,
    \begin{align*}
        \C[X_{n\times m};\Theta_{n\times m'}]^{\bfa,\bfb}&\cong \bigoplus_{\ell(\lambda)\leq n} \left(GL_n^\lambda\right)^{\oplus \abs{\SSMT'(\lambda, \bfa, \bfb)}}.
    \end{align*}
\end{corollary}

\newpage

\bibliography{bibliography}{}
\bibliographystyle{alpha}

\end{document}

%% file: macros.tex
\newcommand{\C}{\mathbb{C}}
\newcommand{\Z}{\mathbb{Z}}

\newcommand{\DD}{\mathcal{D}}
\newcommand{\LL}{\mathcal{L}}

\newcommand{\MP}{{\rm M\!P\!}}

\newcommand{\SYT}{\mathcal{SYT}}
\newcommand{\MT}{\mathcal{MT}}
\newcommand{\SSYT}{\mathcal{SSYT}}
\newcommand{\SSMT}{\mathcal{SSMT}}
\newcommand{\SMT}{\mathcal{SMT}}

\newcommand{\SSsT}{\mathcal{SSST}}
\newcommand{\ORBW}{\mathcal{ORBW}}

\newcommand{\bfa}{{\boldsymbol{a}}}
\newcommand{\bfb}{{\boldsymbol{b}}}

\newcommand{\mo}{\varepsilon}

\newcommand{\multi}[1]{\left\{\mskip-5mu\left\{#1\right\}\mskip-5mu\right\}}
\newcommand{\ov}[1]{\overline{#1}}
\newcommand{\un}[1]{\underline{#1}}
\newcommand{\ovun}[1]{\overline{\underline{#1}}}

\newcommand{\abs}[1]{\lvert #1\rvert}
\newcommand{\defeq}{\vcentcolon=}

\definecolor{defncolor}{HTML}{C44800}
\newcommand{\defn}[1]{\textcolor{defncolor}{\emph{#1}}}

\newcommand{\id}{\epsilon}

\newcommand{\inline}[1]{{\begin{tabular}{c}$#1$\end{tabular}}}

\DeclareMathOperator{\Sym}{Sym}

\DeclareMathOperator{\End}{End}
\DeclareMathOperator{\Hom}{Hom}

\DeclareMathOperator{\sgn}{sgn}
\DeclareMathOperator{\sRSK}{sRSK}
\DeclareMathOperator{\smRSK}{smRSK}

\newcommand{\tT}{{\tilde{T}}}

\newcommand{\tS}{{\tilde{S}}}
\newcommand{\tR}{{\tilde{R}}}
\newcommand{\tB}{{\tilde{B}}}

\newcommand{\tpi}{{\tilde{\pi}}}

\newcommand{\trho}{{\tilde{\rho}}}

\newcommand{\hpi}{{\hat{\pi}}}
\newcommand{\hnu}{{\hat{\nu}}}
\newcommand{\hT}{{\hat{T}}}
\newcommand{\hS}{{\hat{S}}}

\definecolor{c1}{HTML}{0085FF}
\definecolor{c2}{HTML}{FF9000}
\definecolor{c3}{HTML}{90D300}
\definecolor{c4}{HTML}{9728A1}
\newcommand{\colortop}[2]{\foreach \i in {#1}  {\filldraw[fill=#2,draw=#2,line width = 1pt] (T\i) circle (4pt);}}

\newcommand{\colorbot}[2]{\foreach \i in {#1}  {\filldraw[fill=#2,draw=#2,line width = 1pt] (B\i) circle (4pt);}}
\newcommand{\colortopalt}[2]{\foreach \i in {#1}  {\filldraw[fill=white,draw=#2,line width = 1pt] (T\i) circle (4pt);}}
\newcommand{\colorbotalt}[2]{\foreach \i in {#1}  {\filldraw[fill=white,draw=#2,line width = 1pt] (B\i) circle (4pt);}}


%% file: figures.ex_standardized_sp.tex
\newcommand{\hl}[1]{\textcolor{c1}{#1}}

\begin{align*}
        \multi{\multi{1,1,\ov2},\multi{2,\ov2,\un2},\multi{\un1,\ovun2},\multi{\un1,\ovun2}}\\
        \multi{\multi{\hl1,\hl2,\ov2},\multi{2,\ov2,\un2},\multi{\un1,\ovun2},\multi{\un1,\ovun2}}\\
        \multi{\multi{\hl1,\hl2,\ov2},\multi{\hl3,\ov2,\un2},\multi{\un1,\ovun2},\multi{\un1,\ovun2}}\\
        \multi{\multi{\hl1,\hl2,\hl4},\multi{\hl3,\hl5,\un2},\multi{\un1,\ovun2},\multi{\un1,\ovun2}}\\
        \multi{\multi{\hl1,\hl2,\hl4},\multi{\hl3,\hl5,\un2},\multi{\hl{\un1},\ovun2},\multi{\hl{\un2},\ovun2}}\\
        \multi{\multi{\hl1,\hl2,\hl4},\multi{\hl3,\hl{\un3}},\multi{\hl{\un1},\ovun2},\multi{\hl{\un2},\ovun2}}\\
        \{\{\hl1,\hl2,\hl4\},\{\hl3,\hl5,\hl{\un3}\},\{\hl{\un1},\hl{\un4}\},\{\hl{\un2},\hl{\un5}\}\}
\end{align*}

%% file: figures.ex_standardization.tex

\newcommand\hlcell{\Yfillcolour{c3!35}}
\newcommand\whcell{\Yfillcolour{white}}

\begin{align*}
    \begin{array}{ccccc}
       \tT=\inline{\young(11<\ov1>,<2><1\ov2>,<\ov1><1\ov2>)}  &  \inline{\young(!\hlcell12!\whcell<\ov1>,<2><1\ov2>,<\ov1><1\ov2>)}  & \inline{\young(!\hlcell12!\whcell<\ov1>,!\hlcell<5>!\whcell<1\ov2>,<\ov1><1\ov2>)} &
        \inline{\young(!\hlcell126,<5>!\whcell<1\ov2>,!\hlcell7!\whcell<1\ov2>)} & T=\inline{\young(!\hlcell126,5<38>,7<49>)} 
    \end{array}
\end{align*}

%% file: figures.ex_column_dominance.tex
\begin{align*}
    \begin{array}{cccc}
         & \inline{\young(112,<\ov1><1\ov2>,<\ov1><1\ov2>)} & \lhd & \inline{\young(11<\ov1>,2<1\ov2>,<\ov1><1\ov2>)} \\ \vspace{-.1in} \\ \hline \vspace{-.1in} \\
         \alpha^{1} & (1,1,0) & & (1,1,0)\\
         \alpha^{2} & (1,1,1) & & (2,1,0)\\
         \alpha^{\ov1} & (3,1,1) & & (3,1,1)\\
         \alpha^{1\ov2} & (3,3,1) & & (3,3,1)
    \end{array}
\end{align*}

%% file: figures.ex_compositions_and_standardization.tex
\begin{align*}
    \begin{array}{cccc}
         & \inline{\young(11<\ov1>,2<1\ov2>,<\ov1><1\ov2>)} &  & \inline{\young(126,5<38>,7<49>)} \\ \vspace{-.1in} \\ \hline \vspace{-.1in} \\
         & & \alpha^{1} & (1,0,0)\\
         \alpha^{1} & (1,1,0) & \alpha^{2} & (1,1,0)\\
         \alpha^{2} & (2,1,0) & \alpha^{5} & (2,1,0)\\
          &  & \alpha^{6} & (2,1,1)\\
         \alpha^{\ov1}& (3,1,1) & \alpha^{7} & (3,1,1)\\
         & & \alpha^{38} & (3,2,1)\\
         \alpha^{1\ov2} & (3,3,1) & \alpha^{49} & (3,3,1)
    \end{array}
\end{align*}

%% file: figures.example_diagram_basis.tex
\[\begin{array}{c}
        \begin{tikzpicture}[xscale=.5,yscale=.5,line width=1.25pt] 
        \foreach \i in {1,2,3,4,5,6,7}  { \path (\i,1.25) coordinate (T\i); \path (\i,.25) coordinate (B\i); } 
            \filldraw[fill= black!12,draw=black!12,line width=4pt]  (T1) -- (T7) -- (B7) -- (B1) -- (T1);
            \draw[black] (T1)--(T2)--(B1)--(T1);
            \draw[black] (T3)--(B2);
            \draw[black] (T4)--(T5)--(B4)--(T4);
            \draw[black] (T7)--(B7);
            \draw[c1] (B3) .. controls +(0,+.35) and +(0,+.35) .. (B5);
            \colortop{1,...,7}{black}
            \colorbot{1,2,4,7}{black}
            \colorbot{3,5,6}{c1}
        \end{tikzpicture} \\
       \begin{tikzpicture}[xscale=.5,yscale=.5,line width=1.25pt] 
        \foreach \i in {1,2,3,4,5,6,7}  { \path (\i,1.25) coordinate (T\i); \path (\i,.25) coordinate (B\i); } 
            \filldraw[fill= black!12,draw=black!12,line width=4pt]  (T1) -- (T7) -- (B7) -- (B1) -- (T1);
            \draw[black] (T1)--(B1)--(B2)--(T1);
            \draw[black] (B3)--(B5);
            \draw[black] (T7)--(B7)--(B6)--(T7);
            \draw[black] (T2) .. controls +(0,-.35) and +(0,-.35) .. (T4);
            \draw[c1] (T3) .. controls +(0,-.35) and +(0,-.35) .. (T5);
            \colortop{1,2,4,7}{black}
            \colortop{3,5,6}{c1}
            \colorbot{1,...,7}{black}
        \end{tikzpicture}
    \end{array}=\textcolor{c1}{x^2}\begin{array}{c}
        \begin{tikzpicture}[xscale=.5,yscale=.5,line width=1.25pt] 
        \foreach \i in {1,2,3,4,5,6,7}  { \path (\i,1.25) coordinate (T\i); \path (\i,.25) coordinate (B\i); } 
            \filldraw[fill= black!12,draw=black!12,line width=4pt]  (T1) -- (T7) -- (B7) -- (B1) -- (T1);
            \draw[black] (T1)--(T2)--(B2)--(B1)--(T1);
            \draw[black] (T3)--(T5);
            \draw[black] (B3)--(B5);
            \draw[black] (T7)--(B7)--(B6)--(T7);
            \colortop{1,...,7}{black}
            \colorbot{1,...,7}{black}
        \end{tikzpicture}
\end{array}\]

%% file: figures.partition_action.tex
\begin{center}
    \begin{tabular}{cc}
         $\inline{\begin{tikzpicture}[line width=1.25pt, xscale=.7, yscale=.7]
                \path(1.2,1) node {$\youngx(1.5,\;\;<5>,<12><4>,<3>)$};
                \path (-.1,1.1) coordinate (T21);
                \path (.5,1.8) coordinate (T31);
                \path (2.45,.5) coordinate (T13);
                \path (1.45,1.1) coordinate (T22);
                \path (-.75,2.7) coordinate (C1);
                \path (.25,2.7) coordinate (C2);
                \path (1.25,2.7) coordinate (C3);
                \path (2.25,2.7) coordinate (C4);
                \path (3.25,2.7) coordinate (C5);
                \path (-.75,3.2) coordinate (B1);
                \path (.25,3.2) coordinate (B2);
                \path (1.25,3.2) coordinate (B3);
                \path (2.25,3.2) coordinate (B4);
                \path (3.25,3.2) coordinate (B5);
                \path (-.75,4.2) coordinate (T1);
                \path (.25,4.2) coordinate (T2);
                \path (1.25,4.2) coordinate (T3);
                \path (2.25,4.2) coordinate (T4);
                \path (3.25,4.2) coordinate (T5);
                \draw[gray] (T21) .. controls +(-.6,0) and +(+.6,0) .. (C1);
                \draw[gray] (T31) .. controls +(0,+.6) and +(0,-.6) .. (C3);
                \draw[gray] (T22) .. controls +(0,+.6) and +(0,-.6) .. (C4);
                \draw[gray] (T13) .. controls +(0,+.8) and +(0,-.8) .. (C5);
                \draw[black] (C1) -- (C2);
                \draw[black] (T1)--(T2)--(B3);
                \draw[black] (B1)--(B2)--(T3);
                \draw[black] (B4)--(B5)--(T5)--(B4);
                \filldraw[fill=gray,draw=gray,line width = 1pt] (T21) circle (2pt);
                \filldraw[fill=gray,draw=gray,line width = 1pt] (T31) circle (2pt);
                \filldraw[fill=gray,draw=gray,line width = 1pt] (T13) circle (2pt);
                \filldraw[fill=gray,draw=gray,line width = 1pt] (T22) circle (2pt);
                \foreach \i in {C1, C2, C4, B1, B2, B4, T1, T2, T3, B5, C5, C3, B3, T4, T5} {\filldraw[fill=black,draw=black,line width = 1pt] (\i) circle (4pt);}
        \end{tikzpicture}}=\inline{\youngx(1.5,\;\;<4>,<3><5>,<12>)}$ &
         $\inline{\begin{tikzpicture}[line width=1.25pt, xscale=.7, yscale=.7]
                \path(1.2,1) node {$\youngx(1.5,\;\;<5>,<12><4>,<3>)$};
                \path (-.1,1.1) coordinate (T21);
                \path (.5,1.8) coordinate (T31);
                \path (2.45,.5) coordinate (T13);
                \path (1.45,1.1) coordinate (T22);
                \path (-.75,2.7) coordinate (C1);
                \path (.25,2.7) coordinate (C2);
                \path (1.25,2.7) coordinate (C3);
                \path (2.25,2.7) coordinate (C4);
                \path (3.25,2.7) coordinate (C5);
                \path (-.75,3.2) coordinate (B1);
                \path (.25,3.2) coordinate (B2);
                \path (1.25,3.2) coordinate (B3);
                \path (2.25,3.2) coordinate (B4);
                \path (3.25,3.2) coordinate (B5);
                \path (-.75,4.2) coordinate (T1);
                \path (.25,4.2) coordinate (T2);
                \path (1.25,4.2) coordinate (T3);
                \path (2.25,4.2) coordinate (T4);
                \path (3.25,4.2) coordinate (T5);
                \draw[gray] (T21) .. controls +(-.6,0) and +(+.6,0) .. (C1);
                \draw[gray] (T31) .. controls +(0,+.6) and +(0,-.6) .. (C3);
                \draw[gray] (T22) .. controls +(0,+.6) and +(0,-.6) .. (C4);
                \draw[gray] (T13) .. controls +(0,+.8) and +(0,-.8) .. (C5);
                \draw[black] (C1) -- (C2);
                \draw[black] (T1)--(B1);
                \draw[black] (B2)--(B3)--(T2)--(B2);
                \draw[black] (T3)--(B4)--(B5);
                \draw[black] (T4)--(T5);
                \filldraw[fill=gray,draw=gray,line width = 1pt] (T21) circle (2pt);
                \filldraw[fill=gray,draw=gray,line width = 1pt] (T31) circle (2pt);
                \filldraw[fill=gray,draw=gray,line width = 1pt] (T13) circle (2pt);
                \filldraw[fill=gray,draw=gray,line width = 1pt] (T22) circle (2pt);
                \foreach \i in {C1, C2, C4, B1, B2, B4, T1, T2, T3, B5, C5, C3, B3, T4, T5} {\filldraw[fill=black,draw=black,line width = 1pt] (\i) circle (4pt);}
        \end{tikzpicture}}=0$
    \end{tabular}
\end{center}

%% file: figures.ex_straightening.tex

\newcommand\acell{\Yfillcolour{c1!20}}
\newcommand\bcell{\Yfillcolour{c2!20}}
\newcommand\whcell{\Yfillcolour{white}}

\newcommand{\cramp}[1]{\hspace{-.15in}#1\hspace{-.08in}}

\begin{align*}
    T&=\inline{\young(\;\;<17>,2!\bcell4,!\acell8!\bcell<35>,!\acell<69>)}\\
    A&=\{\{8\},\{6,9\}\}\\
    B&=\{\{4\},\{3,5\}\}
\end{align*}

To conveniently write permutations of the boxes of $T$, we label the boxes above the first row by the following standard Young tableau: \[\inline{\young(12,34,5)}.\] Now the permutations of these boxes that leave the columns of $T$ increasing are:
\begin{align*}
    g(A,B)&=\{\id,(3\,4),(2\,4\,3),(4\,5\,3),(2\,4\,5\,3),(2\,3)(4\,5)\}.
\end{align*}

By \Cref{thm:straightening}, we get the following equation (where we draw the tableau $S$ in place of $v_S$):

\begin{align*}
    \inline{\young(\;\;<17>,2!\bcell4,!\acell8!\bcell<35>,!\acell<69>)}\cramp-
    \inline{\young(\;\;<17>,2!\bcell4,!\bcell<35>!\acell8,!\acell<69>)}\cramp+
    \inline{\young(\;\;<17>,2!\bcell<35>,!\bcell4!\acell8,!\acell<69>)}\cramp+
    \inline{\young(\;\;<17>,2!\bcell4,!\bcell<35>!\acell<69>,!\acell8)}\cramp-
    \inline{\young(\;\;<17>,2!\bcell<35>,!\bcell4!\acell<69>,!\acell8)}\cramp+
    \inline{\young(\;\;<17>,2!\acell8,!\bcell4!\acell<69>,!\bcell<35>)}\hspace{-.15in}=0
\end{align*}

%% file: figures.ex_zero_insertion.tex
\newcommand{\vertinsert}[1]{\raisebox{.25em}{$\hspace{.4em}\downarrow\!\!\raisebox{.35em}{${}_{#1}$}$}}

\newcommand{\horinsert}[1]{\raisebox{.2em}{$\stackrel{#1}{\leftarrow}$}}

\newcommand\grcell{\Yfillcolour{gray!50}}
\newcommand\whcell{\Yfillcolour{white}}

\newcommand{\figdesc}[1]{\begin{minipage}{.8in}\raggedright\vspace{.05in}
#1
\end{minipage}}

\begin{tabular}{ccccccc}
     $\raisebox{-.1in}{$\gyoung(;2;3,;3,:<\vertinsert{0}>,:<\textcolor{c1}{1}>)$}$ & $=$ & $\inline{\gyoung(!\grcell;1!\whcell;3,;3::<\horinsert{0}>:<\textcolor{c2}{\!\!\!\!\raisebox{.015in}{2}}>)}$ & = & $\raisebox{-.1in}{$\gyoung(;1;3,!\grcell;2!\whcell,::<\vertinsert{0}>,::<\textcolor{c1}{3}>)$}$ & = & $\inline{\young(13,23)}$\\
      \figdesc{Insert in \textcolor{c1}{first column}} & & \figdesc{Insert in \textcolor{c2}{row above} bump site} & & \figdesc{Insert in \textcolor{c1}{column right of} bump site} & &
\end{tabular}

%% file: figures.ex_one_insertion.tex
\newcommand{\vertinsert}[1]{\raisebox{.25em}{$\hspace{.4em}\downarrow\!\!\raisebox{.35em}{${}_{#1}$}$}}

\newcommand{\horinsert}[1]{\raisebox{.2em}{$\stackrel{#1}{\leftarrow}$}}

\newcommand\grcell{\Yfillcolour{gray!50}}
\newcommand\whcell{\Yfillcolour{white}}

\newcommand{\figdesc}[1]{\begin{minipage}{.7in}\raggedright\vspace{.05in}
#1
\end{minipage}}

\begin{center}
    \begin{tabular}{ccccccccc}
        $\inline{\gyoung(;2;3:<\horinsert{1}>:<\textcolor{c1}{\!\!\!\!\raisebox{.015in}{1}}>,;3)}$ & \hspace{-.3in}$=$ \hspace{-.3in} & $\raisebox{-.1in}{$\gyoung(!\grcell;1!\whcell;3,;3,::<\vertinsert{1}>,::<\textcolor{c2}{2}>)$}$ & \hspace{-.2in}$=$\hspace{-.2in} & $\inline{\gyoung(;1!\grcell;2!\whcell,;3::<\horinsert{1}>:<\textcolor{c1}{\!\!\!\!\raisebox{.015in}{3}}>)}$ & \hspace{-.25in}$=$\hspace{-.25in} & $\inline{\gyoung(;1;2,!\grcell;3!\whcell,:::<\horinsert{1}>:<\textcolor{c1}{\!\!\!\!\raisebox{.015in}{3}}>)}$ & \hspace{-.3in}$=$\hspace{-.1in} & $\inline{\young(12,3,3)}$\\
    
        \figdesc{Insert in \textcolor{c1}{first row}} & & \figdesc{Insert in \textcolor{c2}{column right of} bump site} & & \figdesc{Insert in \textcolor{c1}{row above} bump site} & & \figdesc{Insert in \textcolor{c1}{row above} bump site}
    \end{tabular}
\end{center}

%% file: figures.ex_multiset_insertion.tex
\begin{align*}
    \hpi&=\inline{\begin{tikzpicture}[xscale=.5,yscale=.5,line width=1.25pt] 
                \foreach \i in {1,2,3,4,5,6,7,8}  { \path (\i,1.25) coordinate (T\i); \path (\i,.25) coordinate (B\i); } 
                \filldraw[fill= black!12,draw=black!12,line width=4pt]  (T1) -- (T8) -- (B8) -- (B1) -- (T1);
                \draw[black] (T1) -- (T2);
                \draw[black] (B1) -- (B2);
                \draw[black] (T3) -- (T4) -- (B4);
                \draw[black] (B3)  .. controls +(0,+.5) and +(0,+.5) .. (B5) -- (T6);
                \draw[black] (T5) -- (B7) -- (B8);
                \draw[black] (B6)--(T7);
                \colortop{1,2,3}{c1}
                \colortop{4,5}{c2}
                \colortopalt{6,7}{c3}
                \colortopalt{8}{c4}
                \colorbot{1,2,3}{c1}
                \colorbot{4,5}{c2}
                \colorbotalt{6,7}{c3}
                \colorbotalt{8}{c4}
            \end{tikzpicture}}\\
    w(\hpi)&=\left(\begin{array}{cccc}
         2 & 12 & \ov1 & \ov1\\
         \ov1\ov2 & 2 & \ov1 & 12
    \end{array}\right)\\
    sRSK(w(\hpi))&=\left(\inline{\young(<2><\ov1>,<12><\ov1\ov2>)}\raisebox{-.2in}{,}\inline{\young(<2><\ov1>,<12><\ov1>)}\right)\\
    smRSK(\hpi)&=\left(\inline{\young(\;\;<\dots><11>,<2><\ov1>,<12><\ov1\ov2>)}\raisebox{-.3in}{,}\inline{\young(\;\;<\dots><11><\ov2>,<2><\ov1>,<12><\ov1>)}\right)
\end{align*}

%% file: figures.ex_epa_product.tex
\newcommand{\topdiagram}{\begin{tikzpicture}[xscale=.5,yscale=.5,line width=1.25pt] 
        \foreach \i in {1,2,3,4}  { \path (\i,1.25) coordinate (T\i); \path (\i,.25) coordinate (B\i); } 
        \filldraw[fill= black!12,draw=black!12,line width=4pt]  (T1) -- (T4) -- (B4) -- (B1) -- (T1);
        \draw[black] (B1)--(T1) .. controls +(0,-.4) and +(0,-.4) .. (T3);
        \draw[black] (B2)--(B3);
        \draw[black] (T4)--(B4);
        \colortop{1,2}{c1}
        \colortopalt{3,4}{c3}
        \colorbot{1,2}{c1}
        \colorbotalt{3,4}{c3}
    \end{tikzpicture}}

We compute the product of $\inline{\begin{tikzpicture}[xscale=.3,yscale=.3,line width=1.25pt] 
        \foreach \i in {1,2,3,4}  { \path (\i,1.25) coordinate (T\i); \path (\i,.25) coordinate (B\i); } 
        \filldraw[fill= black!12,draw=black!12,line width=4pt]  (T1) -- (T4) -- (B4) -- (B1) -- (T1);
        \draw[black] (B1)--(T1) .. controls +(0,-.4) and +(0,-.4) .. (T3);
        \draw[black] (B2)--(B3);
        \draw[black] (T4)--(B4);
        \colortop{1,2}{c1}
        \colortopalt{3,4}{c3}
        \colorbot{1,2}{c1}
        \colorbotalt{3,4}{c3}
    \end{tikzpicture}}$ and $\inline{\begin{tikzpicture}[xscale=.3,yscale=.3,line width=1.25pt] 
    \foreach \i in {1,2,3,4}  { \path (\i,1.25) coordinate (T\i); \path (\i,.25) coordinate (B\i); } 
        \filldraw[fill= black!12,draw=black!12,line width=4pt]  (T1) -- (T4) -- (B4) -- (B1) -- (T1);
        \draw[black] (T1) -- (1,.35) .. controls +(0,+.4) and +(0,+.4) .. (B3);
        \draw[black] (T4) -- (4,.35) .. controls +(0,+.4) and +(0,+.4) .. (B2);
        \colortop{1,2}{c1}
        \colortopalt{3,4}{c3}
        \colorbot{1,2}{c1}
        \colorbotalt{3,4}{c3}
    \end{tikzpicture}}$ in $\MP_{(2),(2)}(x)$:

\begin{align*}
    &\frac{1}{2!2!}\left(
    \begin{array}{c}
    \topdiagram \\
    \begin{tikzpicture}[xscale=.5,yscale=.5,line width=1.25pt] 
    \foreach \i in {1,2,3,4}  { \path (\i,1.25) coordinate (T\i); \path (\i,.25) coordinate (B\i); } 
        \filldraw[fill= black!12,draw=black!12,line width=4pt]  (T1) -- (T4) -- (B4) -- (B1) -- (T1);
        \draw[black] (T1) -- (1,.35) .. controls +(0,+.4) and +(0,+.4) .. (B3);
        \draw[black] (T4) -- (4,.35) .. controls +(0,+.4) and +(0,+.4) .. (B2);
        \colortop{1,2}{c1}
        \colortopalt{3,4}{c3}
        \colorbot{1,2}{c1}
        \colorbotalt{3,4}{c3}
    \end{tikzpicture}
    \end{array}\!+\!\begin{array}{c}
    \topdiagram \\
    \begin{tikzpicture}[xscale=.5,yscale=.5,line width=1.25pt] 
    \foreach \i in {1,2,3,4}  { \path (\i,1.25) coordinate (T\i); \path (\i,.25) coordinate (B\i); } 
        \filldraw[fill= black!12,draw=black!12,line width=4pt]  (T1) -- (T4) -- (B4) -- (B1) -- (T1);
        \draw[black] (T2) -- (1,.35) .. controls +(0,+.4) and +(0,+.4) .. (B3);
        \draw[black] (T4) -- (4,.35) .. controls +(0,+.4) and +(0,+.4) .. (B2);
        \colortop{1,2}{c1}
        \colortopalt{3,4}{c3}
        \colorbot{1,2}{c1}
        \colorbotalt{3,4}{c3}
    \end{tikzpicture}
    \end{array}\!-\!\begin{array}{c}
    \topdiagram \\
    \begin{tikzpicture}[xscale=.5,yscale=.5,line width=1.25pt] 
    \foreach \i in {1,2,3,4}  { \path (\i,1.25) coordinate (T\i); \path (\i,.25) coordinate (B\i); } 
        \filldraw[fill= black!12,draw=black!12,line width=4pt]  (T1) -- (T4) -- (B4) -- (B1) -- (T1);
        \draw[black] (T1) -- (1,.35) .. controls +(0,+.4) and +(0,+.4) .. (B3);
        \draw[black] (T3) -- (4,.35) .. controls +(0,+.4) and +(0,+.4) .. (B2);
        \colortop{1,2}{c1}
        \colortopalt{3,4}{c3}
        \colorbot{1,2}{c1}
        \colorbotalt{3,4}{c3}
    \end{tikzpicture}
    \end{array}
    \!-\!\begin{array}{c}
    \topdiagram \\
    \begin{tikzpicture}[xscale=.5,yscale=.5,line width=1.25pt] 
    \foreach \i in {1,2,3,4}  { \path (\i,1.25) coordinate (T\i); \path (\i,.25) coordinate (B\i); } 
        \filldraw[fill= black!12,draw=black!12,line width=4pt]  (T1) -- (T4) -- (B4) -- (B1) -- (T1);
        \draw[black] (T2) -- (1,.35) .. controls +(0,+.4) and +(0,+.4) .. (B3);
        \draw[black] (T3) -- (4,.35) .. controls +(0,+.4) and +(0,+.4) .. (B2);
        \colortop{1,2}{c1}
        \colortopalt{3,4}{c3}
        \colorbot{1,2}{c1}
        \colorbotalt{3,4}{c3}
    \end{tikzpicture}
    \end{array}\right)\\
    =&\hspace{.09in}\frac{1}{4}\left(x\begin{array}{c}
    \begin{tikzpicture}[xscale=.5,yscale=.5,line width=1.25pt] 
    \foreach \i in {1,2,3,4}  { \path (\i,1.25) coordinate (T\i); \path (\i,.25) coordinate (B\i); } 
    \filldraw[fill= black!12,draw=black!12,line width=4pt]  (T1) -- (T4) -- (B4) -- (B1) -- (T1);
    \draw[black] (T3) .. controls +(0,-.4) and +(0,-.4) .. (T1)--(1,.35) .. controls +(0,+.4) and +(0,+.4) .. (B3);
    \draw[black] (T4) -- (4,.35) .. controls +(0,+.4) and +(0,+.4) .. (B2);
    \colortop{1,2}{c1}
    \colortopalt{3,4}{c3}
    \colorbot{1,2}{c1}
    \colorbotalt{3,4}{c3}
    \end{tikzpicture}
    \end{array}
    \!+\!\begin{array}{c}
    \begin{tikzpicture}[xscale=.5,yscale=.5,line width=1.25pt] 
    \foreach \i in {1,2,3,4}  { \path (\i,1.25) coordinate (T\i); \path (\i,.25) coordinate (B\i); } 
    \filldraw[fill= black!12,draw=black!12,line width=4pt]  (T1) -- (T4) -- (B4) -- (B1) -- (T1);
    \draw[black] (T1).. controls +(0,-.4) and +(0,-.4) ..(T3);
    \draw[black] (1,.35) .. controls +(0,+.4) and +(0,+.4) .. (B3);
    \draw[black] (T4) -- (4,.35) .. controls +(0,+.4) and +(0,+.4) .. (B2);
    \colortop{1,2}{c1}
    \colortopalt{3,4}{c3}
    \colorbot{1,2}{c1}
    \colorbotalt{3,4}{c3}
    \end{tikzpicture}
    \end{array}
    \!-\!\begin{array}{c}
    \begin{tikzpicture}[xscale=.5,yscale=.5,line width=1.25pt] 
    \foreach \i in {1,2,3,4}  { \path (\i,1.25) coordinate (T\i); \path (\i,.25) coordinate (B\i); } 
    \filldraw[fill= black!12,draw=black!12,line width=4pt]  (T1) -- (T4) -- (B4) -- (B1) -- (T1);
    \draw[black] (T3).. controls +(0,-.4) and +(0,-.4) ..(T1)--(1,.35) .. controls +(0,+.4) and +(0,+.4) .. (B3);
    \draw[black] (4,.35) .. controls +(0,+.4) and +(0,+.4) .. (B2);
    \colortop{1,2}{c1}
    \colortopalt{3,4}{c3}
    \colorbot{1,2}{c1}
    \colorbotalt{3,4}{c3}
    \end{tikzpicture}
    \end{array}
    \!-\!\begin{array}{c}
    \begin{tikzpicture}[xscale=.5,yscale=.5,line width=1.25pt] 
    \foreach \i in {1,2,3,4}  { \path (\i,1.25) coordinate (T\i); \path (\i,.25) coordinate (B\i); } 
    \filldraw[fill= black!12,draw=black!12,line width=4pt]  (T1) -- (T4) -- (B4) -- (B1) -- (T1);
    \draw[black] (T1).. controls +(0,-.4) and +(0,-.4) ..(T3);
    \draw[black] (B1)--(B4);
    \colortop{1,2}{c1}
    \colortopalt{3,4}{c3}
    \colorbot{1,2}{c1}
    \colorbotalt{3,4}{c3}
    \end{tikzpicture}
    \end{array}\right)
    \intertext{Notice that the last diagram does not correspond to a restricted multiset partition, so it projects to zero by \Cref{lem:projection_nonzero}.}
    =&\hspace{.09in}\frac{1}{4}\left(x\begin{array}{c}
    \begin{tikzpicture}[xscale=.5,yscale=.5,line width=1.25pt] 
    \foreach \i in {1,2,3,4}  { \path (\i,1.25) coordinate (T\i); \path (\i,.25) coordinate (B\i); } 
    \filldraw[fill= black!12,draw=black!12,line width=4pt]  (T1) -- (T4) -- (B4) -- (B1) -- (T1);
    \draw[black] (T3).. controls +(0,-.4) and +(0,-.4) ..(T1)--(1,.35) .. controls +(0,+.4) and +(0,+.4) .. (B3);
    \draw[black] (T4) -- (4,.35) .. controls +(0,+.4) and +(0,+.4) .. (B2);
    \colortop{1,2}{c1}
    \colortopalt{3,4}{c3}
    \colorbot{1,2}{c1}
    \colorbotalt{3,4}{c3}
    \end{tikzpicture}
    \end{array}
    \!+\!\begin{array}{c}
    \begin{tikzpicture}[xscale=.5,yscale=.5,line width=1.25pt] 
    \foreach \i in {1,2,3,4}  { \path (\i,1.25) coordinate (T\i); \path (\i,.25) coordinate (B\i); } 
    \filldraw[fill= black!12,draw=black!12,line width=4pt]  (T1) -- (T4) -- (B4) -- (B1) -- (T1);
    \draw[black] (T1).. controls +(0,-.4) and +(0,-.4) ..(T3);
    \draw[black] (1,.35) .. controls +(0,+.4) and +(0,+.4) .. (B3);
    \draw[black] (T4) -- (4,.35) .. controls +(0,+.4) and +(0,+.4) .. (B2);
    \colortop{1,2}{c1}
    \colortopalt{3,4}{c3}
    \colorbot{1,2}{c1}
    \colorbotalt{3,4}{c3}
    \end{tikzpicture}
    \end{array}
    \!-\!\begin{array}{c}
    \begin{tikzpicture}[xscale=.5,yscale=.5,line width=1.25pt] 
    \foreach \i in {1,2,3,4}  { \path (\i,1.25) coordinate (T\i); \path (\i,.25) coordinate (B\i); } 
    \filldraw[fill= black!12,draw=black!12,line width=4pt]  (T1) -- (T4) -- (B4) -- (B1) -- (T1);
    \draw[black] (T3).. controls +(0,-.4) and +(0,-.4) ..(T1)--(1,.35) .. controls +(0,+.4) and +(0,+.4) .. (B3);
    \draw[black] (4,.35) .. controls +(0,+.4) and +(0,+.4) .. (B2);
    \colortop{1,2}{c1}
    \colortopalt{3,4}{c3}
    \colorbot{1,2}{c1}
    \colorbotalt{3,4}{c3}
    \end{tikzpicture}
    \end{array}\right).
\end{align*}

%% file: figures.ex_mp_action.tex
\begin{center}
         \scalebox{.7}{\inline{\begin{tikzpicture}[line width=1.25pt, xscale=.7, yscale=.7]
                \path(2.2,1) node {$\youngx(1.5,\;\;<2>,<1\ov1><2>,<1\ov1>)$};
                \path (.9,1.1) coordinate (T21);
                \path (.9,1.8) coordinate (T31);
                \path (3.5,.5) coordinate (T13);
                \path (2.5,1.1) coordinate (T22);
                \path (-.75,2.7) coordinate (C1);
                \path (.25,2.7) coordinate (C2);
                \path (1.25,2.7) coordinate (C3);
                \path (2.25,2.7) coordinate (C4);
                \path (3.25,2.7) coordinate (C5);
                \path (4.25,2.7) coordinate (C6);
                \path (-.75,3.2) coordinate (B1);
                \path (.25,3.2) coordinate (B2);
                \path (1.25,3.2) coordinate (B3);
                \path (2.25,3.2) coordinate (B4);
                \path (3.25,3.2) coordinate (B5);
                \path (4.25,3.2) coordinate (B6);
                \path (-.75,4.2) coordinate (T1);
                \path (.25,4.2) coordinate (T2);
                \path (1.25,4.2) coordinate (T3);
                \path (2.25,4.2) coordinate (T4);
                \path (3.25,4.2) coordinate (T5);
                \path (4.25,4.2) coordinate (T6);
                \draw[gray] (T21) .. controls +(-.6,0) and +(+.6,0) .. (C1);
                \draw[gray] (T31) .. controls +(0,+.6) and +(0,-.6) .. (C2);
                \draw[gray] (T22) .. controls +(0,+.6) and +(0,-.6) .. (C3);
                \draw[gray] (T13) .. controls +(0,+.8) and +(0,-.8) .. (C4);
                \draw[black] (C1) .. controls +(0,-.6) and +(0,-.6) .. (C6);
                \draw[black] (C2) .. controls +(0,-.4) and +(0,-.4) .. (C5);
                \draw[black] (B1) -- (T1) -- (T2);
                \draw[black] (B2) -- (T3);
                \draw[black] (B3) .. controls +(0,+.4) and +(0,+.4) .. (B4);
                \draw[black] (B6) -- (T6);
                \filldraw[fill=gray,draw=gray,line width = 1pt] (T21) circle (2pt);
                \filldraw[fill=gray,draw=gray,line width = 1pt] (T31) circle (2pt);
                \filldraw[fill=gray,draw=gray,line width = 1pt] (T13) circle (2pt);
                \filldraw[fill=gray,draw=gray,line width = 1pt] (T22) circle (2pt);
                \foreach \i in {C1, C2, T1, T2, B1, B2} {\filldraw[fill=c1,draw=c1,line width = 1pt] (\i) circle (4pt);}
                \foreach \i in {C3, C4, T3, T4, B3, B4} {\filldraw[fill=c2,draw=c2,line width = 1pt] (\i) circle (4pt);}
                \foreach \i in {C5, C6, T5, T6, B5, B6} {\filldraw[fill=white,draw=c3,line width = 1pt] (\i) circle (4pt);}
        \end{tikzpicture}}
        +\inline{\begin{tikzpicture}[line width=1.25pt, xscale=.7, yscale=.7]
                \path(2.2,1) node {$\youngx(1.5,\;\;<2>,<1\ov1><2>,<1\ov1>)$};
                \path (.9,1.1) coordinate (T21);
                \path (.9,1.8) coordinate (T31);
                \path (3.5,.5) coordinate (T13);
                \path (2.5,1.1) coordinate (T22);
                \path (-.75,2.7) coordinate (C1);
                \path (.25,2.7) coordinate (C2);
                \path (1.25,2.7) coordinate (C3);
                \path (2.25,2.7) coordinate (C4);
                \path (3.25,2.7) coordinate (C5);
                \path (4.25,2.7) coordinate (C6);
                \path (-.75,3.2) coordinate (B1);
                \path (.25,3.2) coordinate (B2);
                \path (1.25,3.2) coordinate (B3);
                \path (2.25,3.2) coordinate (B4);
                \path (3.25,3.2) coordinate (B5);
                \path (4.25,3.2) coordinate (B6);
                \path (-.75,4.2) coordinate (T1);
                \path (.25,4.2) coordinate (T2);
                \path (1.25,4.2) coordinate (T3);
                \path (2.25,4.2) coordinate (T4);
                \path (3.25,4.2) coordinate (T5);
                \path (4.25,4.2) coordinate (T6);
                \draw[gray] (T21) .. controls +(-.6,0) and +(+.6,0) .. (C1);
                \draw[gray] (T31) .. controls +(0,+.6) and +(0,-.6) .. (C2);
                \draw[gray] (T22) .. controls +(0,+.6) and +(0,-.6) .. (C3);
                \draw[gray] (T13) .. controls +(0,+.8) and +(0,-.8) .. (C4);
                \draw[black] (C1) .. controls +(0,-.6) and +(0,-.6) .. (C6);
                \draw[black] (C2) .. controls +(0,-.4) and +(0,-.4) .. (C5);
                \draw[black] (B2) -- (T1) -- (T2);
                \draw[black] (B1) -- (T3);
                \draw[black] (B3) .. controls +(0,+.4) and +(0,+.4) .. (B4);
                \draw[black] (B6) -- (T6);
                \filldraw[fill=gray,draw=gray,line width = 1pt] (T21) circle (2pt);
                \filldraw[fill=gray,draw=gray,line width = 1pt] (T31) circle (2pt);
                \filldraw[fill=gray,draw=gray,line width = 1pt] (T13) circle (2pt);
                \filldraw[fill=gray,draw=gray,line width = 1pt] (T22) circle (2pt);
                \foreach \i in {C1, C2, T1, T2, B1, B2} {\filldraw[fill=c1,draw=c1,line width = 1pt] (\i) circle (4pt);}
                \foreach \i in {C3, C4, T3, T4, B3, B4} {\filldraw[fill=c2,draw=c2,line width = 1pt] (\i) circle (4pt);}
                \foreach \i in {C5, C6, T5, T6, B5, B6} {\filldraw[fill=white,draw=c3,line width = 1pt] (\i) circle (4pt);}
        \end{tikzpicture}}
        -\inline{\begin{tikzpicture}[line width=1.25pt, xscale=.7, yscale=.7]
                \path(2.2,1) node {$\youngx(1.5,\;\;<2>,<1\ov1><2>,<1\ov1>)$};
                \path (.9,1.1) coordinate (T21);
                \path (.9,1.8) coordinate (T31);
                \path (3.5,.5) coordinate (T13);
                \path (2.5,1.1) coordinate (T22);
                \path (-.75,2.7) coordinate (C1);
                \path (.25,2.7) coordinate (C2);
                \path (1.25,2.7) coordinate (C3);
                \path (2.25,2.7) coordinate (C4);
                \path (3.25,2.7) coordinate (C5);
                \path (4.25,2.7) coordinate (C6);
                \path (-.75,3.2) coordinate (B1);
                \path (.25,3.2) coordinate (B2);
                \path (1.25,3.2) coordinate (B3);
                \path (2.25,3.2) coordinate (B4);
                \path (3.25,3.2) coordinate (B5);
                \path (4.25,3.2) coordinate (B6);
                \path (-.75,4.2) coordinate (T1);
                \path (.25,4.2) coordinate (T2);
                \path (1.25,4.2) coordinate (T3);
                \path (2.25,4.2) coordinate (T4);
                \path (3.25,4.2) coordinate (T5);
                \path (4.25,4.2) coordinate (T6);
                \draw[gray] (T21) .. controls +(-.6,0) and +(+.6,0) .. (C1);
                \draw[gray] (T31) .. controls +(0,+.6) and +(0,-.6) .. (C2);
                \draw[gray] (T22) .. controls +(0,+.6) and +(0,-.6) .. (C3);
                \draw[gray] (T13) .. controls +(0,+.8) and +(0,-.8) .. (C4);
                \draw[black] (C1) .. controls +(0,-.6) and +(0,-.6) .. (C6);
                \draw[black] (C2) .. controls +(0,-.4) and +(0,-.4) .. (C5);
                \draw[black] (B1) -- (T1) -- (T2);
                \draw[black] (B2) -- (T3);
                \draw[black] (B3) .. controls +(0,+.4) and +(0,+.4) .. (B4);
                \draw[black] (B5) -- (T6);
                \filldraw[fill=gray,draw=gray,line width = 1pt] (T21) circle (2pt);
                \filldraw[fill=gray,draw=gray,line width = 1pt] (T31) circle (2pt);
                \filldraw[fill=gray,draw=gray,line width = 1pt] (T13) circle (2pt);
                \filldraw[fill=gray,draw=gray,line width = 1pt] (T22) circle (2pt);
                \foreach \i in {C1, C2, T1, T2, B1, B2} {\filldraw[fill=c1,draw=c1,line width = 1pt] (\i) circle (4pt);}
                \foreach \i in {C3, C4, T3, T4, B3, B4} {\filldraw[fill=c2,draw=c2,line width = 1pt] (\i) circle (4pt);}
                \foreach \i in {C5, C6, T5, T6, B5, B6} {\filldraw[fill=white,draw=c3,line width = 1pt] (\i) circle (4pt);}
        \end{tikzpicture}}
        -\inline{\begin{tikzpicture}[line width=1.25pt, xscale=.7, yscale=.7]
                \path(2.2,1) node {$\youngx(1.5,\;\;<2>,<1\ov1><2>,<1\ov1>)$};
                \path (.9,1.1) coordinate (T21);
                \path (.9,1.8) coordinate (T31);
                \path (3.5,.5) coordinate (T13);
                \path (2.5,1.1) coordinate (T22);
                \path (-.75,2.7) coordinate (C1);
                \path (.25,2.7) coordinate (C2);
                \path (1.25,2.7) coordinate (C3);
                \path (2.25,2.7) coordinate (C4);
                \path (3.25,2.7) coordinate (C5);
                \path (4.25,2.7) coordinate (C6);
                \path (-.75,3.2) coordinate (B1);
                \path (.25,3.2) coordinate (B2);
                \path (1.25,3.2) coordinate (B3);
                \path (2.25,3.2) coordinate (B4);
                \path (3.25,3.2) coordinate (B5);
                \path (4.25,3.2) coordinate (B6);
                \path (-.75,4.2) coordinate (T1);
                \path (.25,4.2) coordinate (T2);
                \path (1.25,4.2) coordinate (T3);
                \path (2.25,4.2) coordinate (T4);
                \path (3.25,4.2) coordinate (T5);
                \path (4.25,4.2) coordinate (T6);
                \draw[gray] (T21) .. controls +(-.6,0) and +(+.6,0) .. (C1);
                \draw[gray] (T31) .. controls +(0,+.6) and +(0,-.6) .. (C2);
                \draw[gray] (T22) .. controls +(0,+.6) and +(0,-.6) .. (C3);
                \draw[gray] (T13) .. controls +(0,+.8) and +(0,-.8) .. (C4);
                \draw[black] (C1) .. controls +(0,-.6) and +(0,-.6) .. (C6);
                \draw[black] (C2) .. controls +(0,-.4) and +(0,-.4) .. (C5);
                \draw[black] (B2) -- (T1) -- (T2);
                \draw[black] (B1) -- (T3);
                \draw[black] (B3) .. controls +(0,+.4) and +(0,+.4) .. (B4);
                \draw[black] (B5) -- (T6);
                \filldraw[fill=gray,draw=gray,line width = 1pt] (T21) circle (2pt);
                \filldraw[fill=gray,draw=gray,line width = 1pt] (T31) circle (2pt);
                \filldraw[fill=gray,draw=gray,line width = 1pt] (T13) circle (2pt);
                \filldraw[fill=gray,draw=gray,line width = 1pt] (T22) circle (2pt);
                \foreach \i in {C1, C2, T1, T2, B1, B2} {\filldraw[fill=c1,draw=c1,line width = 1pt] (\i) circle (4pt);}
                \foreach \i in {C3, C4, T3, T4, B3, B4} {\filldraw[fill=c2,draw=c2,line width = 1pt] (\i) circle (4pt);}
                \foreach \i in {C5, C6, T5, T6, B5, B6} {\filldraw[fill=white,draw=c3,line width = 1pt] (\i) circle (4pt);}
        \end{tikzpicture}}}
\end{center}

\newcommand{\X}{.2in}

\begin{center}
       \scalebox{.8}{ =\hspace{.3in}\inline{\youngx(1.5,\;\;<\ov1>,<2><2>,<11\ov1>)}
        \hspace{\X}+\hspace{\X}\inline{\youngx(1.5,\;\;<\ov1>,<11><2>,<2\ov1>)}
        \hspace{\X}-\hspace{\X}\inline{\youngx(1.5,\;\;<\ov1>,<2\ov1><2>,<11>)}
        \hspace{\X}-\hspace{\X}\inline{\youngx(1.5,\;\;<\ov1>,<11\ov1><2>,<2>)}}
\end{center}

\begin{align*}
       =2\left(\inline{\youngx(1.5,\;\;<\ov1>,<2><2>,<11\ov1>)}+\inline{\youngx(1.5,\;\;<\ov1>,<11><2>,<2\ov1>)}\right)
\end{align*}